\documentclass[12pt]{amsart}

\usepackage{subfiles}
\usepackage{euscript}
\usepackage{tabu}
\usepackage[margin=0.8in]{geometry}
\usepackage[dvipsnames]{xcolor}
\usepackage{hyperref}
\usepackage{graphicx}			
\usepackage{amssymb}
\usepackage{mathrsfs}
\usepackage{amsthm}
\usepackage{amsmath}
\usepackage{stmaryrd}
\usepackage{fancyhdr}
\usepackage{svg}

\usepackage{tikz}
\usepackage{tikz-cd}
\usetikzlibrary{arrows,arrows.meta}
\usetikzlibrary{arrows.meta,        
                decorations.pathmorphing,
                matrix}%
\tikzset{
  commutative diagrams/.cd, 
  arrow style=tikz, 
  diagrams={>=cm to}
}

\usepackage{accents}
\usepackage{upgreek}
\usepackage{enumerate}
\usepackage{bm}
\usepackage{mathtools}
\usepackage[all]{xy}
\usepackage{caption}

\usepackage{url}
\usepackage{float}
\usepackage{todonotes}
\usepackage{bbm}
\usepackage{colonequals}
\usepackage{longtable}

\usepackage[full]{textcomp}
\usepackage[cal=cm]{mathalfa}
\usepackage{xparse}
\usepackage{comment}
\usepackage{cite}

\NewDocumentCommand{\Cartesiansquare}{m m m m O{} O{} O{} O{}}{
\begin{equation*}\begin{tikzcd}[ampersand replacement=\&]
  #1 \arrow{r}{#5} \arrow{d}{#6} \arrow[dr, phantom, "\square"] \& #2 \arrow{d}{#7} \\
  #3 \arrow{r}{#8} \& #4
\end{tikzcd} \end{equation*}}
\newtheorem{thm}{Theorem}[section]
\newtheorem*{thm*}{Theorem}
\newtheorem*{prop*}{Proposition}

\newtheorem{innercustomthm}{Theorem}
\newenvironment{customthm}[1]
  {\renewcommand\theinnercustomthm{#1}\innercustomthm}
  {\endinnercustomthm}
  
\newtheorem{innercustomcor}{Corollary}
\newenvironment{customcor}[1]
  {\renewcommand\theinnercustomcor{#1}\innercustomcor}
  {\endinnercustomcor}
  
\newtheorem{cor}[thm]{Corollary}
\newtheorem{prop}[thm]{Proposition}
\newtheorem{lem}[thm]{Lemma}

\theoremstyle{definition}
\newtheorem{defn}[thm]{Definition}

\newtheorem{exmp}[thm]{Example}

\newtheorem{rem}[thm]{Remark}
\newtheorem*{rem*}{Remark}

\usepackage[normalem]{ulem}

\newcommand\Mbar{\overline{\mathcal{M}}}
\newcommand\Mfr{\mathfrak{M}}
\newcommand\VZ{\mathcal{VZ}}

\newcommand\rad{\mathrm{rad}}

\newcommand{\M}{\mathcal{M}}
\renewcommand{\rad}{\mathrm{rad}}
\newcommand{\C}{\mathbb{C}}
\newcommand{\tgamma}{\tilde{\Gamma}}
\newcommand{\Ob}{\mathrm{Ob}}
\newcommand{\T}{\mathbf{T}}
\renewcommand{\P}{\mathbb{P}}
\newcommand{\K}{\mathcal{K}}

\title{The dual complex of $\mathcal{M}_{1,n}(\mathbb{P}^r,d)$ via the geometry of the Vakil--Zinger moduli space}

\author[S. Kannan]{Siddarth Kannan}\address{Department of Mathematics, Massachusetts Institute of Technology}
\email{\url{spkannan@mit.edu}}
\author[T. Song]{Terry Dekun Song}\address{Department of Pure Mathematics and Mathematical Statistics, University of Cambridge, Cambridge, CB3 0WA}\email{\url{ds2016@cam.ac.uk}}

\begin{document}

\maketitle

\pagestyle{myheadings}
\markright{\hfill Dual complex of $\mathcal{M}_{1,n}(\mathbb{P}^r,d)$\hfill}

\begin{abstract}
We study normal crossings compactifications of the moduli space of maps $\M_{g, n}(\mathbb{P}^r, d)$, for $g = 0$ and $g = 1$. In each case we explicitly determine the dual boundary complex, and prove that it admits a natural interpretation as a moduli space of decorated metric graphs. We prove that the dual complexes are contractible when $r \geq 1$ and $d > g$. When $g = 1$, our result depends on a new understanding of the connected components of boundary strata in the Vakil--Zinger desingularization and its modular interpretation by Ranganathan--Santos-Parker--Wise.
\end{abstract}
\section*{Introduction}

In this paper we are interested in compactifications of the moduli space $\M_{g, n}(\mathbb{P}^r, d)$ parameterizing degree $d$ maps from smooth $n$-pointed curves of genus $g$ to projective space. When $g = 0$, the Kontsevich moduli space $\Mbar_{0, n}(\mathbb{P}^r, d)$ provides a smooth modular normal crossings compactification of $\M_{0, n}(\mathbb{P}^r, d)$. When $g = 1$, a smooth modular normal crossings compactification $\widetilde{\M}_{1, n}(\mathbb{P}^r, d)$ of $\M_{1, n}(\mathbb{P}^r, d)$ has been constructed using logarithmic and tropical geometry by Ranganathan--Santos-Parker--Wise \cite{rspw} following earlier work of Vakil--Zinger \cite{vakilzinger} and Hu--Li \cite{huli}. 

We investigate both combinatorial and geometric aspects of the mapping spaces $\widetilde{\M}_{1,n}(\mathbb{P}^r,d)$. Each point in this moduli space is a stable map satisfying a linear condition at certain nodes of the domain curve, and has an associated \textit{radially-aligned} dual graph $(\mathbf{G}, \rho)$ (Definition \ref{defn-radalign}). The smoothing of maps leads to specialization relations among these graphs. These can be packaged into a space $\Delta_{1, n}(d)$, which is a \textit{symmetric $\Delta$-complex} in the sense of \cite[\S 3]{cgp}. We write $\widetilde{\M}(\mathbf{G}, \rho)$ for the locus of maps with dual graph $(\mathbf{G}, \rho)$.

\begin{customthm}{A}\label{thm:determination}
Fix $r \geq 1$. The stratum $\widetilde{\M}(\mathbf{G}, \rho)$ is irreducible, and there is an explicit combinatorial criterion for when $\widetilde{\M}(\mathbf{G}, \rho)$ is nonempty. The dual complex of the divisor
\[\widetilde{\M}_{1, n}(\mathbb{P}^r, d) \smallsetminus \M_{1, n}(\mathbb{P}^r, d) \]
is naturally identified with the symmetric $\Delta$-complex $\Delta_{1, n}(d)$.
\end{customthm}

 We also prove that the dual complexes are contractible.
\begin{customthm}{B} \label{thm:mainthm}
    Fix $r\geq 1$. The dual complex $\Delta_{0, n}(d)$ of \[\Mbar_{0, n}(\mathbb{P}^r, d) \smallsetminus \M_{0, n}(\mathbb{P}^r, d) \] for $d>0$ as well as the dual complex $\Delta_{1, n}(d)$ of 
    \[\widetilde{\M}_{1, n}(\mathbb{P}^r, d) \smallsetminus \M_{1, n}(\mathbb{P}^r, d)   \] for $d>1$ are contractible. In particular, the reduced homology groups of the dual complexes vanish.
\end{customthm}

If $\mathcal{X}$ is a smooth Deligne--Mumford stack and $\overline{\mathcal{X}}$ is a normal crossings compactification of $\mathcal{X}$, the dual complex $\Delta(\mathcal{D})$ is a symmetric $\Delta$-complex\footnote{The definition of a symmetric $\Delta$-complex generalizes taht of a cell complex and will be recalled in Section \ref{sec:deltavir}.} encoding the combinatorics of the divisor $\mathcal{D} = \overline{\mathcal{X}} \smallsetminus \mathcal{X}$. It has a $p$-cell for each codimension $p$ stratum of $\mathcal{D}$, and these cells are glued together according to inclusions of strata. Self-gluing of cells is allowed, reflecting the possibility that the fundamental groups of strata may have non-trivial monodromy actions on the branches of $\mathcal{D}$. The simple homotopy type of the complex $\Delta(\mathcal{D})$ is independent of the choice of normal crossings compactification \cite{stepanov, Thuillier2007, pay, harper}, and is hence an invariant of $\mathcal{X}$. The connection between the dual complex and mixed Hodge theory is well-known: if $\mathcal{X}$ has dimension $n$ over $\mathbb{C}$, Deligne's weight spectral sequence gives isomorphisms \cite[Theorem 5.8]{cgp} \[\widetilde{H}_{k-1}(\Delta(\mathcal{D}); \mathbb{Q} )\cong \mathrm{Gr}_{2n}^WH^{2n-k}(\mathcal{X};\mathbb{Q})\cong (W_0H_c^{k}(\mathcal{X};\mathbb{Q}))^\vee.\] Taking the reduced homology of the dual complex, we deduce the vanishing of the top-weight singular cohomology of the mapping spaces.
\begin{customcor}{C}\label{cor:maincor}
For $\mathcal{X}=\mathcal{M}_{0,n}(\mathbb{P}^r,d)$ with $d>0$ or $\mathcal{X} = \mathcal{M}_{1,n}(\mathbb{P}^r,d)$ with $d>1$, \[\mathrm{Gr}_{2\delta}^W H^{2\delta-k}(\mathcal{X}; \mathbb{Q})\cong (W_0 H_c^{k}(\mathcal{X};\mathbb{Q}))^\vee=0\]
for all $k$, where $\delta = \dim_{\mathbb{C}}(\mathcal{X})$.
\end{customcor}

 \begin{rem*}
 Note that when $d = 0$, the mapping spaces are $\mathcal{M}_{0,n}\times \mathbb{P}^r$ and $\mathcal{M}_{1,n}\times \mathbb{P}^r$, both of which have non-trivial top-weight cohomology \cite{robinsonwhitehouse, Vogtmann_1990}, \cite[Corollary 1.3]{aluffi_topology_2022}. When $d = 1$, we have $\M_{1,n}(\mathbb{P}^r,1)= \varnothing$. The dual complexes do not depend on $r$, as long as $r \geq 1$.
 \end{rem*}
 
 One can view Theorem \ref{thm:mainthm} as stating that the input of a non-trivial target destroys all top-weight cohomology classes. In our work, the vanishing of the top-weight cohomology is witnessed by the combinatorics of boundary strata.

\subsection*{Compactifications of $\M_{g, n}(\mathbb{P}^r, d)$} The mapping space $\mathcal{M}_{g,n}(\mathbb{P}^r,d)$ admits a natural forgetful map
\[\mathcal{M}_{g,n}(\mathbb{P}^r,d)\to \mathcal{M}_{g,n}\] given by $(C,f)\mapsto C$, which is smooth for $d > g$. However, a modular normal crossings compactification of $\mathcal{M}_{g,n}(\mathbb{P}^r,d)$ for all genera is not known. One compactification is provided by the Kontsevich space of stable maps \[\mathcal{M}_{g,n}(\mathbb{P}^r,d)\subset \Mbar_{g,n}(\mathbb{P}^r,d).\]
Some of the key properties of this space are summarized as follows:
\begin{itemize}
    \item \textit{(Lack of) smoothness}: For $g\geq 1$, the space $\Mbar_{g,n}(\mathbb{P}^r,d)$ is typically reducible with components of dimension greater than $\dim \mathcal{M}_{g,n}(\mathbb{P}^r,d)$. It satisfies Murphy's law as stated by Vakil \cite{murphy}. When $g = 0$ the space is smooth, and it is a normal crossings compactification of $\M_{0, n}(\mathbb{P}^r, d)$ \cite[§1.3.2]{Kon95}.
    \item \textit{Explicit boundary combinatorics}: The complement $\Mbar_{g,n}(\mathbb{P}^r,d)\smallsetminus \mathcal{M}_{g,n}(\mathbb{P}^r,d)$ is stratified by what we call $(g,n,d)$-graphs: these are prestable dual graphs decorated with degree assignments on vertices.
    \item \textit{Explicit boundary geometry}: A decorated dual graph specifies a stratum of $\Mbar_{g,n}(\mathbb{P}^r,d)$ given by fiber products of mapping spaces $\mathcal{M}_{g',n'}(\mathbb{P}^r,d')$ along evaluation maps to $\mathbb{P}^r$.
\end{itemize}
The failure of $\Mbar_{g, n}(\mathbb{P}^r, d)$ to be a smooth normal crossings compactification suggests that new tools are needed to understand the weight filtration of $\M_{g, n}(\mathbb{P}^r, d)$ in positive genus. In genus one, we use the moduli space $\widetilde{\M}_{1, n}(\mathbb{P}^r, d)$ of \textit{radially-aligned} stable maps which satisfy a factorization condition, defined in \cite[Definition 4.1]{rspw}. It is constructed as a closed subscheme of an iterated blow-up of the Kontsevich space $\Mbar_{1, n}(\mathbb{P}^r, d)$, and can be understood as a further blow-up of the Vakil--Zinger desingularization $\mathcal{VZ}_{1, n}(\mathbb{P}^r, d)$. The space $\widetilde{\M}_{1, n}(\mathbb{P}^r, d)$ enjoys the following properties:
\begin{itemize}
    \item \textit{Smoothness}: The space $\widetilde{\M}_{1, n}(\mathbb{P}^r, d)$ is smooth, proper, and contains $\M_{1, n}(\mathbb{P}^r, d)$ as a dense open subset. The complement of $\M_{1, n}(\mathbb{P}^r, d)$ in $\widetilde{\M}_{1, n}(\mathbb{P}^r, d)$ is a divisor with normal crossings.
    \item \textit{Explicit boundary combinatorics}: The boundary divisor $\widetilde{\M}_{1, n}(\mathbb{P}^r, d) \smallsetminus \mathcal{M}_{1,n}(\mathbb{P}^r,d)$ is stratified by \textit{radially-aligned} $(1,n,d)$-graphs. These are a subset of the dual graphs indexing the strata of the Kontsevich compactification, endowed with the additional decoration of a \textit{radial alignment}: a surjective function from the irreducible components of the curve to $\{0, \ldots, k\}$ for some $k \geq 0$, satisfying some natural combinatorial conditions.
    \item \textit{Tractable boundary geometry}: A decorated dual graph specifies a stratum of $\widetilde{\M}_{1, n}(\mathbb{P}^r, d)$. The resulting strata are not as explicit as that of the Kontsevich space, but it is still possible to probe their geometry: each stratum is given by a torus bundle over a certain closed subscheme of a graph stratum in the Kontsevich space. A key step towards Theorem \ref{thm:determination} is our proof that these strata are irreducible.
\end{itemize}

\subsection*{Contractibility of the virtual dual complex}
In Section \ref{sec:dualcomplex}, following the construction of the moduli space of tropical curves $\Delta_{g,n}$ as in \cite{aluffi_topology_2022}, we construct a moduli space $\Delta_{g,n}^{\mathrm{vir}}(d)$ of pairs $(\mathbf{G}, \ell)$ where $\mathbf{G}$ is the degree-decorated dual graph of a Kontsevich stable map, and $\ell$ is a metric on the edges of $\mathbf{G}$. The space $\Delta_{g,n}^{\mathrm{vir}}(d)$ is a symmetric $\Delta$-complex which encodes the combinatorics of Kontsevich stable maps. However, it is not the genuine dual complex of $\mathcal{M}_{g,n}(\mathbb{P}^r,d)$ unless $g=0$, since stable map compactifications are far from being normal crossings. We hence call $\Delta_{g,n}^{\mathrm{vir}}(d)$ the \textit{virtual dual complex} of the Kontsevich space. Our main result on the topology of $\Delta_{g,n}^{\mathrm{vir}}(d)$ is the construction of an explicit deformation retract to a point.

\begin{thm*}[Theorem \ref{thm:virtual-contraction}]
    For $d>0$, there is a deformation retract 
    \[\Delta_{g,n}^{\mathrm{vir}}(d) \times [0, 1] \to \Delta_{g,n}^{\mathrm{vir}}(d) \] onto a point. In other words, $\Delta^{\mathrm{vir}}_{g,n}(d)$ is contractible.
\end{thm*}

When $g = 0$, the virtual dual complex $\Delta_{0, n}^{\mathrm{vir}}(d)$ coincides with the genuine dual complex of the boundary divisor in the Kontsevich moduli space $\Mbar_{0, n}(\mathbb{P}^r, d)$, so Theorem \ref{thm:virtual-contraction} proves the genus-zero case of Theorem \ref{thm:mainthm}. Identifying contractible subcomplexes has been a natural technique in understanding the topology of closely related dual complexes, such as that of $\M_{0,n}$ in \cite[Theorem 2.4]{robinsonwhitehouse}, $\M_{g, n}$ in \cite[§4]{aluffi_topology_2022} and that of $\mathcal{H}_{g,n}$ in \cite[§5]{bck24}. The construction of the deformation retraction will be adapted to prove that the genuine dual complex $\Delta_{1, n}(d)$ of the boundary divisor in $\widetilde{\M}_{1, n}(\mathbb{P}^r, d)$ is contractible.

\begin{prop*}[Proposition \ref{prop:subspace_preserved}]
The dual complex $\Delta_{1, n}(d)$ maps homeomorphically onto a subspace of $\Delta_{1, n}^{\mathrm{vir}}(d)$. This subspace is preserved by the deformation retract of Theorem \ref{thm:virtual-contraction}.
\end{prop*}
The genus-one case of Theorem \ref{thm:mainthm} then follows from Proposition \ref{prop:subspace_preserved}.

\subsection*{Related work}
The dual complex $\Delta_{g, n}$ of the Deligne--Mumford compactification \[\mathcal{M}_{g,n}\subset \Mbar_{g,n}\] has been the subject of intense study in recent years \cite{cgp}, \cite{aluffi_topology_2022}, \cite{cgpsn}. It is naturally interpreted as a moduli space of tropical curves \cite{acp}, \cite{ccuw}, and has rich topology: it has led to the discovery of many more non-algebraic cohomology classes on $\mathcal{M}_{g, n}$ than were previously known.

It is natural to extend this line of inquiry to moduli spaces in the vicinity of $\mathcal{M}_{g,n}$ for which modular normal crossings compactifications are available. This has proven fruitful for moduli spaces of abelian varieties \cite{topweightAg} and moduli spaces of pointed hyperelliptic curves \cite{bck24}. Other potential directions include moduli of abelian differentials on $\mathcal{M}_{g,n}$ \cite{BCGGM, cghms} with its multi-scale and logarithmic compactifications, and universal Picard groups with compactifications given by compactified \cite{melo1} or logarithmic Jacobians \cite{mw}.

Since the moduli spaces mentioned above all admit natural maps to $\mathcal{M}_{g,n},$ we comment on how they interact with the top-weight cohomology groups of $\mathcal{M}_{g,n}.$ In general, a morphism $X\to Y$ induces cohomology pullback on the weight-graded pieces $\mathrm{Gr}^W_{j}H^\star(Y)\to \mathrm{Gr}^W_{j}H^\star(X).$ In particular, when $Y$ is equidimensional, the pullback on top-weight cohomology takes the form $$\mathrm{Gr}^W_{2\dim Y}H^\star(Y)\to \mathrm{Gr}^W_{2\dim Y}H^\star(X).$$ Therefore, when $\dim Y\neq \dim X,$ as in the case of $\mathcal{M}_{g,n}(\mathbb{P}^r,d)\to \mathcal{M}_{g,n}$ for $r>0,$ the pullback map on ordinary cohomology does not send the top-weight cohomology group of $Y$ to that of $X.$ On the other hand, when $X\to Y$ is proper, the proper pullback $W_0 H_c^\star(Y)\to W_0 H_c^\star(X)$ does map between the weight-zero compactly-supported cohomology groups, which are dual to the top weight ordinary cohomology when $X$ and $Y$ are both smooth, by Poincar\'e duality. For instance, this is the case of the universal Picard group $\mathrm{{Pic}}^d_{g,n}\to \mathcal{M}_{g,n},$
However, the morphism $\M_{g, n}(\mathbb{P}^r, d) \to \M_{g,n}$ is not proper when $d > 0$, so there is no way of pulling back compactly-supported cohomology classes.

As far as the topology of mapping spaces, work of Farb--Wolfson \cite{farbwolf} determines the rational cohomology of $\M_{0, n}(\mathbb{P}^r, d)$ with its weight filtration: the genus-zero case of Corollary \ref{cor:maincor} can be deduced from their work, which is influenced by earlier work of Segal \cite{segal} on mapping spaces in genus zero. A recursive algorithm to compute the Betti numbers of $\Mbar_{0, n}(\mathbb{P}^r, d)$ has been given by Getzler--Pandharipande \cite{getzpand}. In their work they also determine the virtual Poincare polynomial of $\M_{0, n}(\mathbb{P}^r, d)$ \cite[Theorem 5.6]{getzpand}. From their calculation it is possible to deduce that the weight zero compactly supported Euler characteristic of $\M_{0, n}(\mathbb{P}^r, d)$ vanishes, as is also implied by Theorem \ref{thm:mainthm}. We computed the $S_n$-equivariant topological Euler characteristic of $\Mbar_{1,n}(\mathbb{P}^r, d)$ in \cite{ksdih}.

The intersection theories of $\M_{0, n}(\mathbb{P}^r, d)$ and $\Mbar_{0, n}(\mathbb{P}^r, d)$ have been studied by Pandharipande \cite{nonlineargrassmannian, Qdivisors} and Behrend--O'Halloran \cite{BOH}. Tautological rings of these spaces have been studied by Oprea \cite{opretorus}. The homology groups of $\Mbar_{g, n}(\mathbb{P}^r, d)$ have been shown to satisfy a form of representation stability \cite{tosteson}, which gives restrictions on the asymptotic behavior of $H_i(\Mbar_{g, n}(\mathbb{P}^r, d); \mathbb{Q})$ as $n \to \infty$. 
\subsection*{Further questions} 
We investigate the dual complexes of two classes of mapping spaces to projective space. There are three natural directions for generalization.
\begin{itemize}
    \item \textit{Genus}: The moduli space $\mathcal{M}_{2,n}(\mathbb{P}^r,d)$ of maps from genus two curves to $\mathbb{P}^r$ admits a modular normal crossings compactification in the work of Battistella and Carocci \cite{battistellacarocci}. As their work builds on the perspective of \cite{rspw}, it would be interesting to describe the strata and the dual complex of their compactification.

    \item \textit{Target}: As we shall see in Section \ref{sec:dualcomplex}, the deformation retraction of the virtual dual complex can be adapted to any (smooth, projective) target that has Picard rank one. It would be interesting to apply this contractibility argument to the boundary complexes of mapping spaces to other targets such as Grassmannians.
    \item \textit{Computation}: The description of the dual complex can be seen as an enrichment of the top-weight cohomology $\mathrm{Gr}^W_{2\delta}H^k(\mathcal{M}_{1,n}(\mathbb{P}^r,d))$ and $\mathrm{Gr}^W_{2\delta}H^k(\mathcal{M}_{0,n}(\mathbb{P}^r,d))$. Another direction of extending the results is to pursue the full cohomology of the spaces $\Mbar_{0,n}(\mathbb{P}^r,d)$, $\widetilde{\M}_{1,n}(\mathbb{P}^r,d)$ and their interiors. This is being carried out in the ongoing work \cite{ds} of the second author. 
\end{itemize}
\subsubsection*{Acknowledgements} We are grateful to Dhruv Ranganathan for encouraging this project and for several helpful conversations. We also benefited from conversations with Melody Chan and Navid Nabijou. We would like to acknowledge comments from anonymous referees which helped us to improve the exposition of the paper. SK is supported by NSF DMS-2401850. TS is supported by a Cambridge Trust international scholarship.

\section{Combinatorial preliminaries}\label{sec:combinatorics}
In this section we establish some graph-theoretic definitions which we will refer back to throughout the paper. The reader may consider skipping this section at first and referring back as necessary, especially if they are familiar with the stratification of the Kontsevich moduli space $\Mbar_{g, n}(\mathbb{P}^r, d)$ by dual graphs.
\begin{defn}\label{defn:gnd-graph}
Let $g, n, d \geq 0$. A $(g, n, d)$-graph is a tuple $\mathbf{G} = (G, w, \delta, m)$ where:
\begin{itemize}
\item $G$ is a connected graph;
\item $w : V(\mathbf{G}) \to \mathbb{Z}_{\geq 0}$ is called the \textit{genus function};
\item $\delta: V(\mathbf{G}) \to \mathbb{Z}_{\geq 0}$ is called the \textit{degree function};
\item $m: \{1, \ldots, n\} \to V(\mathbf{G})$ is called the \textit{marking function}. 
\end{itemize}
These data are required to satisfy:
\begin{enumerate}[(1)]
\item $\dim_{\mathbb{Q}} H_1(G, \mathbb{Q}) + \sum_{v \in V(\mathbf{G})} w(v) = g$;
\item $\sum_{v \in V(\mathbf{G})} \delta(v) = d$.
\end{enumerate}
A $(g, n, d)$-graph is called \textit{stable} if for all vertices $v \in V(\mathbf{G})$ with $\delta(v) = 0$, we have
\[ 2 w(v) - 2 + \mathrm{val}(v) + |m^{-1}(v)| > 0, \]
where $\mathrm{val}(v)$ means the graph valence of $v$. 
\end{defn}

\begin{figure}[h]
    \centering
    \includegraphics[scale=1.3]{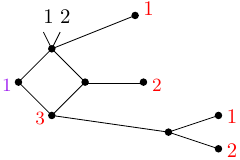}
    \caption{An example of a stable $(2, 2, 9)$-graph $\mathbf{G}$. Red numbers next to vertices indicate the value of $\delta$ on the vertex; other vertices are assumed to have $\delta$-value equal to $0$. The vertex with a purple $1$ is assumed to have $w$-value $1$, while other vertices have $w = 0$. Finally, the function $m$ is visualized by adding labeled half-edges.}
    \label{fig:vir-comb-type-exmp}
\end{figure}

The collection $\Gamma^{\mathrm{ps}}_{g, n}(d)$ of all $(g,n,d)$-graphs forms a category\footnote{Here ``ps" stands for prestable.}, where morphisms are given by compositions of isomorphisms and weighted edge contractions in the sense of \cite[\S 2.2]{cgp}. Particularly important will be the subcategory $\Gamma_{g, n}(d)$ of all stable $(g, n, d)$-graphs. 

\begin{rem}\label{rem-label}
    It can be useful to consider markings $m: M\to V(\mathbf{G})$ from an arbitrary finite set $M$, without reference to $\{1, \ldots, n\}$.
\end{rem}

\begin{rem} We informally explain the morphisms in this category. Contracting an edge $e$ in $\mathbf{G}$ consists of shrinking $e$ to a point, in order to make a new graph $\mathbf{G}/e$. The degree, weight, and marking functions are defined as:\begin{itemize}
    \item If $e$ is not a loop, the two vertices $v_1$ and $v_2$ incident to $e$ are combined into a new vertex $v'$. Then $$\delta(v') = \delta(v_1) + \delta(v_2), w(v') = w(v_1) + w(v_2),$$ and $m^{-1}(v') = m^{-1}(v_1) \cup m^{-1}(v_2)$.
    \item If $e$ is a loop, we increase $w$ on the vertex supporting $e$ by $1$, so as to preserve condition (1) in Definition \ref{defn:gnd-graph}. The degree and marking functions are retained.
\end{itemize} Isomorphisms of $(g, n, d)$-graphs are isomorphisms of the underlying graphs that respect the genus, degree, and marking functions.\end{rem}  

In this paper, we will mainly work with $(g, n, d)$-graphs when $g \leq 1$. In the genus one case, we will often make reference to additional decorations that we now introduce.


\subsection{Radial alignments on genus $1$ graphs}\label{subsec-radial}
The \textit{core} of a $(1, n, d)$-graph is a central notion in this paper.
\begin{defn}
Let $\mathbf{G}$ be a $(1, n, d)$-graph. The \textit{core} of $\mathbf{G}$ is the minimal subgraph of genus $1$, where the genus of a subgraph $G' \subseteq G$ is defined by
\[ \dim_{\mathbb{Q}} H_1(G';\mathbb{Q}) + \sum_{v \in V(G')} w(v). \]
\end{defn}

The set of edges in the core is denoted as $C(\mathbf{G})$, and we use $T(\mathbf{G}):=E(\mathbf{G})\smallsetminus C(\mathbf{G})$ to denote its complement. The set of vertices in the core is denoted by $V^{\mathrm{core}}(\mathbf{G})$ and its complement by $V^{\mathrm{tree}}(\mathbf{G})$. Given a $(1, n, d)$-graph $\mathbf{G}$, we get a rooted tree by contracting all of the edges in the core. A rooted tree has a canonical directed tree structure, where all edges are directed away from the root. Therefore, on $\mathbf{G}$, there is:
\begin{enumerate}
\item a canonical way to orient all edges in $T(\mathbf{G})$;
\item a canonical partial order $<$ on the set of vertices in $V^{\mathrm{tree}}(\mathbf{G})$, which is extended to $V(\mathbf{G})$ by declaring that $v_1 < v_2$ if $v_1$ is in the core and $v_2$ is not.
\end{enumerate}

\begin{rem}\label{rem:vtreeT}
    There is a bijection between $V^{\mathrm{tree}}(\mathbf{G})$ and $T(\mathbf{G})$ given by assigning each vertex $v\in V^{\mathrm{tree}}(\mathbf{G})$ the unique edge $e_{v}$ that connects $v$ to $v'\in V(\mathbf{G})$ with $v'<v.$ Because the contraction of $\mathbf{G}$ by its core is a connected rooted tree, there is a unique such vertex $v'$ and the edge $e_v$ is unique.
\end{rem}

Formally, if $v_1, v_2 \in V^{\mathrm{tree}}(\mathbf{G})$, we have $v_1 < v_2$ if and only if there is a directed path from $v_1$ to $v_2$. This partial order allows us to define \textit{radial alignments}, which are the key combinatorial tool in constructing a modular normal crossings compactification of the space $\M_{1, n}(\mathbb{P}^r, d)$. Informally, we think of a radial alignment as declaring an ordering of the vertices in $V^{\mathrm{tree}}(\mathbf{G})$ by their distance to the core.
\begin{defn}\label{defn-radalign}
Let $\mathbf{G}$ be a $(1, n, d)$-graph. A \textit{radial alignment} of $\mathbf{G}$ is a surjective function $\rho: V(\mathbf{G}) \to \{0, \ldots, k\}$ for some $k \geq 0$ such that
\begin{itemize}
\item $\rho^{-1}(0) = V^{\mathrm{core}}(\mathbf{G})$
\item if $v < w$, then $\rho(v) < \rho(w)$. 
\end{itemize}
\end{defn}

\begin{defn}
An $n$-marked \textit{radially-aligned graph} of genus $1$, degree $d$, and length $k$ is a pair $(\mathbf{G}, \rho)$ where $\mathbf{G} = (G, w, \delta, m)$ is a $(1, n, d)$-graph and $\rho : V(\mathbf{G}) \to \{0, \ldots, k\}$ is a radial alignment of $\mathbf{G}$. We say that $(\mathbf{G}, \rho)$ is \textit{stable} if $\mathbf{G}$ is stable as a $(1, n, d)$-graph. An \textit{isomorphism} $\psi: (\mathbf{G}, \rho) \to (\mathbf{G}', \rho')$ is an isomorphism of the underlying $(1, n, d)$-graphs such that the diagram
\[ \begin{tikzcd}
V(\mathbf{G}) \arrow[d, "\psi"] \arrow[r, "\rho"] & \{0, \ldots, k\} \\
V(\mathbf{G}') \arrow[ur, "\rho'"']  & 
\end{tikzcd}\]
commutes.
\end{defn}

\begin{figure}[h]
    \centering
    \includegraphics[scale = 1.2]{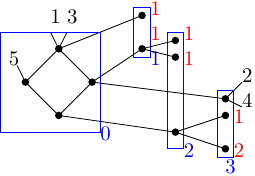}
    \caption{A stable radially-aligned $5$-marked graph $\mathbf{G}$ of genus $1$, degree $7$, and length $3$. All vertices have $g(v) = 0$, and the red labels indicate the $\delta$ degree of a vertex. The blue boxes indicate the fibers of the surjection $V(\mathbf{G}) \to \{0, 1, 2, 3\}$, and the labeled half-edges indicate the marking function $m$.}
    \label{fig:comb-type-exmp}
\end{figure}

Given an $n$-marked radially-aligned graph $\mathbf{G}$, there are several ways to subdivide its edges.

\begin{defn}\label{subdivision}
Let $\mathbf{G}$ be a stable $n$-marked radially-aligned graph of genus $1$, degree $d$, and length $k$.
\begin{enumerate}
\item We define the \textit{canonical subdivision} $\hat{\mathbf{G}}$ of $\mathbf{G}$ as follows: if $e \in E(\mathbf{G})$ is an edge outside the core, such that $e$ is directed from $v_1$ to $v_2$, with $\rho(v_1) = i$ and $\rho(v_2) = j$ with $i < j$, we subdivide $e$ by adding $j - i - 1$ bivalent vertices, and setting the genus and $\delta$-degree of each new vertex to be $0$.
\item For a fixed index $r \in \{1, \ldots, k\}$, we define the \textit{subdivision at radius $r$}: for each edge $e$ outside the core of $\mathbf{G}$, if $e$ is a directed from $v_1$ to $v_2$ with $\rho(v_1) < r$ and $\rho(v_2) > r$, we add one bivalent vertex to $e$, setting the genus and $\delta$-degree of each new vertex to be $0$. Also, for each marking $j \in \{1, \ldots, n\}$ such that $\rho(m(j)) < r$, we add a new vertex $v_j$ connected by an edge $e_j$ to the vertex $m(j)$, and then change the marking function so that it sends $j$ to $v_j$. We write $\hat{\mathbf{G}}_r$ for the resulting graph; it has a radial alignment $\rho_r$ which is defined to be the same as $\rho$ on the vertices which were already in $\mathbf{G}$, and such that $\rho_r(v_{\mathrm{new}}) = r$ for all new vertices $v_{\mathrm{new}}$ added to $\mathbf{G}$.
\end{enumerate}
\end{defn}

\begin{rem}
Note the asymmetry in the definitions of the canonical subdivision $\hat{\mathbf{G}}$ and the subdivision $\hat{\mathbf{G}}_r$ at radius $r$: in $\hat{\mathbf{G}}_r$, half-edges corresponding to marked points are subdivided, whereas they are left unchanged in the canonical subdivision $\hat{\mathbf{G}}$. 

Informally, these processes can be visualized as follows: for the subdivision at radius $r$, draw a circle around the core of $\mathbf{G}$ so that it intersects each vertex $v \in V(\mathbf{G})$ with $\rho(v) = r$, and so that all vertices with $\rho(v) < r$ are inside the circle, and all vertices with $\rho(v) > r$ are outside the circle. In this set-up, half-edges corresponding to marked points are understood to have infinite length. Then the circle intersects $\mathbf{G}$ once for every half-edge corresponding to markings $j$ with $\rho(m(j)) < r$ and once for every edge $e$ directed from $v_1$ to $v_2$ with $\rho(v_1) < r$ and $\rho(v_2) > r$, and bivalent genus-zero vertices are added at these points of intersection. 

The canonical subdivision can be visualized similarly, except that half-edges are not subdivided. See Figures \ref{fig:comb-type-subdiv} and \ref{fig:radius-subdiv} for examples of carrying out these procedures on the radially-aligned stable graph $\mathbf{G}$ from Figure \ref{fig:comb-type-exmp}.
\end{rem}

\begin{figure}
    \centering
    \includegraphics[scale=1.3]{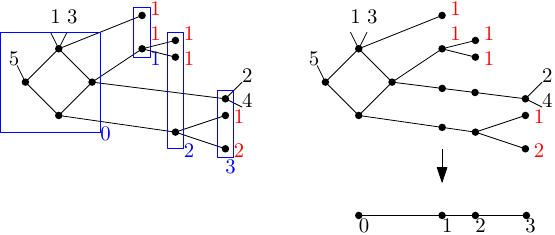}
    \caption{The canonical subdivision $\hat{\mathbf{G}}$ of the graph $\mathbf{G}$ from Figure \ref{fig:comb-type-exmp}, along with the map $\hat{\rho}: \hat{\mathbf{G}} \to P_{3}$ induced by the radial alignment.}
    \label{fig:comb-type-subdiv}
\end{figure}

\begin{figure}
    \centering
    \includegraphics[scale=1.3]{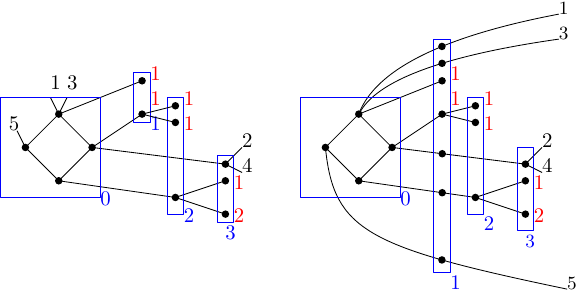}
    \caption{The subdivision $\hat{\mathbf{G}}_1$ at radius $1$ of the graph $\mathbf{G}$ from Figure \ref{fig:comb-type-exmp}, along with the induced radial alignment.}
    \label{fig:radius-subdiv}
\end{figure}

Note that $\hat{\mathbf{G}}$ and $\hat{\mathbf{G}}_r$ are still radially-aligned $n$-marked graphs of genus $1$, but they are no longer stable. The significance of $\hat{\mathbf{G}}_r$ will become clear when we discuss the relevant moduli problem. The conceptual significance of the definition of the \textit{canonical} subdivision $\hat{\mathbf{G}}$ is that it allows us to view the radial alignment $\rho$ on $\mathbf{G}$ as a map of graphs $\hat{\rho}: \hat{\mathbf{G}} \to P_{k}$, where $P_k$ is a path with $k$ edges, whose vertices are labelled by $\{0, \ldots, k\}$ in the natural order from left to right. Here $\hat{\rho}$ contracts the core of $\mathbf{G}$ to the leftmost vertex of $P_k$, and maps edges outside of the core to edges of $P_k$. The value of this perspective will become clear momentarily, when we define \textit{radial merges} below. Radial merges are the analogue of the standard edge contractions in the setting of radially-aligned graphs.

\begin{defn}
Let $\mathbf{G} = (G, w, m, \delta, \rho)$ be a stable $n$-marked radially-aligned graph of genus $1$ and degree $d$. Suppose $\mathbf{G}$ has length $k$. Given an integer $i \in \{1, \ldots, k\}$, define the \textit{radial merge} of $\mathbf{G}$ along $i$ as follows:
\begin{enumerate}

\item post-compose $\rho$ with the surjection $\{0, \ldots, k\} \to \{0, \ldots, k - 1\}$ which fixes all $j<i$ and decreases all $j \geq i$ by $1$;

\item whenever $v, w \in V(\mathbf{G})$ with $v \in f^{-1}(i - 1)$ and $w \in f^{-1}(i)$, such that there is an edge $e$ between $v$ and $w$, perform the edge contraction of $e$.  
\end{enumerate}

The collection of all stable $n$-marked radially-aligned graphs of genus $1$ and degree $d$ form a category $\Gamma_{1, n}^\rad(d)$, where the morphisms are compositions of isomorphisms, contractions of edges in the core, and the radial merges. Similarly, the collection of all radially-aligned $(1,n,d)$-graphs (not necessarily stable) forms a category $\Gamma_{1, n}^{\rad,\mathrm{ps}}(d)$ with the same classes of morphisms.
\end{defn}

An intuitive way to understand radial merges is that they are determined by pulling back edge contractions of the path $P_k$ under the map $\hat{\mathbf{G}} \to P_k$ from the canonical subdivision induced by the radial alignment. The radial merge along $i \in \{1, \ldots, k\}$ can be visualized as follows: contract the edge between vertices $i - 1$ and $i$ of $P_k$, and contract all edges in its preimage to obtain an edge contraction of $\hat{\mathbf{G}}$. The radial merge of $\mathbf{G}$ along $i$ is obtained by removing all bivalent, degree zero vertices in the resulting graph.

\begin{figure}[h]
    \centering
    \includegraphics[scale=1.2]{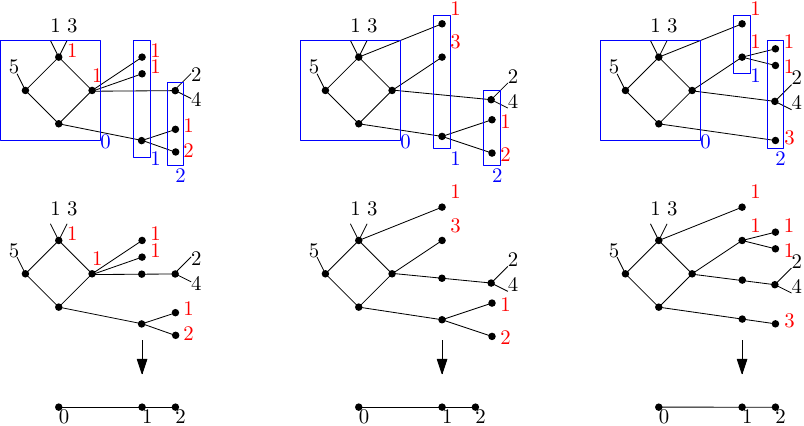}
    \caption{From left to right: the radial merges of the graph $\mathbf{G}$ from Figure \ref{fig:comb-type-exmp} along $1$, $2$, and $3$, above the corresponding canonical subdivisions and maps to $P_2$.}
    \label{fig:rad-merge-exmp}
\end{figure}

The final combinatorial notion that will be important for us is that of a \textit{contraction radius} associated to a radially-aligned graph when $d > 0$.
\begin{defn}\label{def:minimal-radius}
Given an $n$-marked radially-aligned stable graph $(\mathbf{G}, \rho) $ of genus $1$, degree $d > 0$, and length $k$, define the \textit{contraction radius} of $(\mathbf{G},\rho)$ by
\[ \mathrm{rad}(\mathbf{G}, \rho) := \min \left\{ j \in \{0, \ldots, k\} \mid \sum_{v \in \rho^{-1}(j)} \delta(v) \neq 0 \right\}, \]
and define the \textit{degree of the contraction radius} by 
\[ d_{\mathrm{min}}(\mathbf{G}, \rho):= \sum_{v \in \rho^{-1}(\mathrm{rad}(\mathbf{G}, \rho))} \delta(v).  \]
Write $\tilde{\Gamma}_{1, n}^\rad(d)$ for the full subcategory of $\Gamma_{1, n}^\rad(d)$ defined by those graphs with $d_{\mathrm{min}}(\mathbf{G}, \rho) > 1$.
\end{defn}
\begin{rem}\label{rem:face_closure}
   An important fact about the subcategory $\tilde{\Gamma}_{1, n}^\rad(d)$ of Definition \ref{def:minimal-radius} is that it is closed under taking core edge contractions and radial merges. This is because the quantity $d_{\mathrm{min}}(\mathbf{G}, \rho)$ does not decrease under these operations.
\end{rem}

\section{Modular compactifications of mapping spaces}\label{sec:compactifications}
In this section, we motivate and define the Vakil--Zinger type genus one mapping space $\widetilde{\M}_{1,n}(\mathbb{P}^r,d)$ as constructed in \cite{rspw}\footnote{The space $\widetilde{\M}_{1,n}(\mathbb{P}^r,d)$ is termed as $\VZ_{1,n}(\mathbb{P}^r,d)$ in \cite{rspw}. As explained in \S4.6 of loc. lit., it is in general not identical to the construction in \cite{vakilzinger}. Therefore, we have chosen the notation $\widetilde{\M}_{1,n}(\mathbb{P}^r,d)$ to avoid confusion.}. These mapping spaces are constructed to \textit{desingularize} the main component of the \textit{Kontsevich stable map} space, which we now recall.

\begin{defn}
An \textit{$n$-marked stable map to $\mathbb{P}^r$} is a map $f: (C, p_1, \ldots, p_n) \to \mathbb{P}^r$, such that $C$ is a proper connected curve with at worst nodal singularities, and $p_i \in C$ are distinct smooth points, for $i = 1, \ldots, n$. We impose the stability condition that on all irreducible components $T \subset C$ such that $T$ is contracted by $f$, we have
\[2 g(T) - 2 + |T \cap \overline{C \smallsetminus T}| + |\{i \mid p_i \in T \}| > 0, \]
where $g(T)$ denotes the arithmetic genus of $T$. The \textit{genus} of the map $f$ is the arithmetic genus of $C$, and the degree of the map is the unique integer $d$ such that $f_*[C] = dL$, where $[C]$ is the fundamental class of the curve and $L \in H_2(\mathbb{P}^r; \mathbb{Z})$ is the class of a line. We write
\[ \Mbar_{g, n}(\mathbb{P}^r, d) \]
for the moduli space of all $n$-pointed stable maps to $\mathbb{P}^r$ of genus $g$ and degree $d$. 
\end{defn}

Recording the degree and genus assignments on the domain curves gives a stratification of $\Mbar_{g, n}(\mathbb{P}^r, d)$ indexed by $(g,n,d)$-graphs. However, unlike the compactification $\M_{g,n}\subset \overline{\M}_{g,n}$, strata dimensions of $(g,n,d)$-graphs can be \textit{larger} than the dimension of the interior $\mathcal{M}_{g,n}(\mathbb{P}^r,d)$, as shown by the example below. This is a combinatorial source of pathologies of the stable maps compactification $\Mbar_{g,n}(\mathbb{P}^r,d)$.

\begin{defn}
    Given a $(g,n,d)$-graph $\mathbf{G},$ we use $\mathcal{K}(\mathbf{G})$ to denote the locally closed stratum in $\Mbar_{g,n}(\P^r,d)$ associated to $\mathbf{G}.$ 
\end{defn}

\begin{exmp}
Let $d>1$, and let $\mathbf{G}_{1,n,d}^{\mathrm{sp}}$ be the $(1,n,d)$-graph consisting of a single genus one vertex supporting all $n$ markings and degree zero, connected to $d$ copies of genus-zero, degree 1 vertices, each by a single edge. The stratum $\mathcal{K}(\mathbf{G}_{1,n,d}^{\mathrm{sp}})$ in the stable map space is a finite quotient of $$\mathcal{M}_{1,n+d}(r,0)\times_{\mathbb{P}^r}\left( \mathcal{M}_{0,1}(\mathbb{P}^r,1)\right)^d,$$ which has dimension $n + r + d(r+2)$. On the other hand, the interior $\mathcal{M}_{1,n}(\mathbb{P}^r,d)$ has dimension $n + d(r+1)$.
\end{exmp}

Informally, the work \cite{rspw} reduces the dimension of such boundary strata by imposing constraints on certain tangent vectors of maps---the so-called \textit{factorization property}. However, more combinatorial data is needed to specify where the tangency constraints are imposed. In turn, parametrizing the extra tropical data leads to a refinement of the combinatorial types $(1,n,d)$-graphs by \textit{radially-aligned} $(1,n,d)$-graphs. 

On the level of spaces, the refinement induces \textit{strata blow-ups} of the stable maps space, and the desired mapping space is identified as the closed subscheme in the blow-up cut out by the constraints. After explaining the motivation of the construction, we now turn to the technical details.

\subsection{radially-aligned prestable curves}\label{sec:mfrad}
Let $\Mfr_{1,n}$ be the stack of prestable curves with genus one and $n$ marked points. It is stratified by dual graphs of genus one and $n$ marked points, which we identify as (not necessarily stable) $(1,n,0)$-graphs. On a combinatorial level, forgetting the radial alignment defines a functor $\Gamma^{\rad,\mathrm{ps}}_{1,n}(0)\to \Gamma^{\mathrm{ps}}_{1,n}(0)$. In particular, this gives a functor on the level of their underlying partially ordered sets.

Following \cite[§6]{ccuw}, this functor induces a morphism of \textit{Artin fans} $$\mathcal{A}_{\Gamma^{\rad,\mathrm{ps}}_{1,n}(0)}\to \mathcal{A}_{\Gamma^{\mathrm{ps}}_{1,n}(0)}.$$ On the other hand, the $(1,n,0)$-stratification of $\Mfr_{1,n}$ is the stratification that underlies a \textit{logarithmic structure} on $\Mfr_{1,n}$. Therefore, there is a morphism of algebraic stacks $\Mfr_{1,n}\to \mathcal{A}_{\Gamma^{\mathrm{ps}}_{1,n}(0)}$.

\begin{defn}\label{defn-radps}
    The moduli stack of radially-aligned prestable curves $\Mfr_{1,n}^{\rad}$ is defined as the fiber product \Cartesiansquare{\Mfr_{1,n}^{\rad}}{\Mfr_{1,n}}{\mathcal{A}_{\Gamma^{\mathrm{rad},\mathrm{ps}}_{1,n}(0)}}{\mathcal{A}_{\Gamma^{\mathrm{ps}}_{1,n}(0)}}
\end{defn}

\begin{rem}\label{pathlength}
    It is helpful to recall that up to the graph automorphisms $\mathrm{Aut}(\mathbf{G}),$ the (Artin) cones in the Artin fan $\mathcal{A}_{\Gamma^{\mathrm{ps}}_{1,n}(0)}$ are of the form $\left(\mathbb{A}^1/\mathbb{G}_{m}\right)^{|E(\mathbf{G})|}$ for each $\mathbf{G} \in \mathrm{Ob}(\Gamma_{1,n}(0))$ glued along morphisms in $\Gamma^{\mathrm{ps}}_{1,n}(0)$. Via the dictionary between Artin fans and cone stacks\footnote{As explained in  \cite[§2]{ccuw}, cone stacks are closely related to generalized cone complexes from \cite{acp}.} in \cite[§6]{ccuw}, each such Artin cone corresponds to a polyhedral cone of the form $\mathbb{R}_{\geq 0}^{|E(\mathbf{G})|}$ in the cone stack $\Sigma_{\Gamma^{\mathrm{ps}}_{1,n}(0)}$ associated to the Artin fan $\mathcal{A}_{\Gamma^{\mathrm{ps}}_{1,n}(0)}$.

    Each non-core vertex $x\in V^{\mathrm{tree}}(\mathbf{G})$ gives rise to a piecewise linear function on the cone $\mathsf{dis}_x: \mathbb{R}_{\geq 0}^{|E(\mathbf{G})|}\to \mathbb{R}_{\geq 0}$ that measures the distance of $x$ from the core with the edge lengths: $$\ell\mapsto\sum_{e\in P} \ell(e),$$ where $P$ is the unique minimal path connecting $x$ to the core.
        
    As explained in \cite[Proposition 3.3.4]{rspw}, the morphism $\mathcal{A}_{\Gamma^{\rad,\mathrm{ps}}_{1,n}(0)}\to \mathcal{A}_{\Gamma^{\mathrm{ps}}_{1,n}(0)}$ is locally the toric blow-up that is induced by subdividing the cones $\mathbb{R}_{\geq 0}^{|E(\mathbf{G})|}$ along the loci where a tie of distances takes place $$\{\ell\in \mathbb{R}_{\geq 0}^{|E(\mathbf{G})|}: \mathsf{dis}_x = \mathsf{dis}_y\}_{x,y\in V^{\mathrm{tree}}(\mathbf{G})}.$$ In other words, the subdivision is the minimal one such that each cone in the subdivision has an unambiguous ordering of the functions $\{\mathsf{dis}_x\}_{x\in V^{\mathrm{tree}}(\mathbf{G})}$. The unambiguous ordering is precisely the \textit{radial alignment} data introduced in the previous section.
\end{rem}

 For a down-to-earth understanding of the stack $\Mfr_{1, n}^{\rad}$, the following definition is useful. Write $\Mfr(\mathbf{G})$ for the locally closed stratum of $\Mfr_{1, n}$ of curves with dual graph equal to $\mathbf{G}$.

\begin{defn}\label{nodes} Suppose $(C, p_1, \ldots, p_n) \in \Mfr(\mathbf{G})$ where $\mathbf{G}$ is a prestable $(1, n, 0)$-graph, let $v \in V^{\mathrm{tree}}(\mathbf{G})$, and let $e \in T(\mathbf{G})$ be the unique edge that connects $v$ to some vertex $w$ such that $w < v$ in the canonical partial order on $V(\mathbf{G})$. On the level of curves, let $C_v$ and ${\mu}_{v}$ be the component and node associated to $v$ and $e_{v}$, respectively. \end{defn}

\begin{figure}[h]
    \centering
    \includegraphics[scale=1.3]{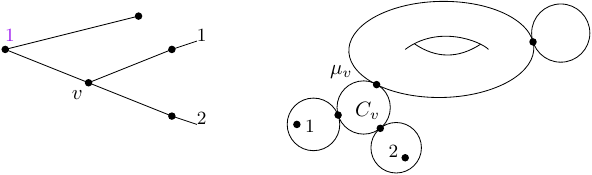}
    \caption{A $2$-marked genus one prestable curve $C$, a vertex $v$ of its dual graph, and the corresponding irreducible component $C_v$ and node ${\mu}_v$.}
    \label{fig:node-fig}
\end{figure}

A point of $\Mfr_{1, n}^{\rad}$ can be thought of as a tuple $(C, p_1, \ldots, p_n, \rho, L_1, \ldots, L_k)$ where $(C, p_1, \ldots, p_n)$ is an $n$-pointed prestable curve of genus $1$, $\rho: V(\mathbf{G}) \to \{0, \ldots, k\}$ is a radial alignment of the dual graph $\mathbf{G}$ of $C$, and
\[L_i \subset \bigoplus_{v \in \rho^{-1}(i)} T_{{\mu}_v} C_v \]
is a line which is not contained in any coordinate subspace. This description follows from Lemma \ref{lem-curvetorus}.

These lines $L_i$ determine contractions $\eta_i: \tilde{C}_i \to \overline{C}_i$, where $\tilde{C}_i \to C$ is a logarithmic modification with dual graph $\hat{\mathbf{G}}_i$ (Definition \ref{subdivision}), and $\eta_i$ contracts those components of $\tilde{C}_i$ corresponding to vertices with $\rho(v) < i$ to an elliptic singularity. These contractions to elliptic singularities are discussed in more detail in ~\cite[\S 3]{rspw}.


\subsection{radially-aligned stable maps} Let $\Mbar_{1,n}(\mathbb{P}^r,d)\to \Mfr_{1,n}$ be the forgetful map, and let $\Mfr_{1,n}^\rad\to \Mfr_{1,n}$ be the logarithmic modification as described above. The space $\Mbar^{\rad}_{1,n}(\mathbb{P}^r,d)$ is defined by the fiber product \Cartesiansquare{\Mbar^{\rad}_{1,n}(\mathbb{P}^r,d)}{\Mbar_{1,n}(\mathbb{P}^r,d)}{\Mfr_{1,n}^\rad}{\Mfr_{1,n}}

The strata of $\Mbar^{\rad}_{1,n}(\mathbb{P}^r,d)$ are indexed by pairs $(\mathbf{G},\rho)$ for a stable $(1,n,d)$-graph $\mathbf{G}$ together with a radial alignment $\rho$. We also recall that the degree labeling gives the \textit{contraction radius} (Definition \ref{def:minimal-radius}) $\mathrm{rad}(\mathbf{G}, \rho)$ and the associated subdivision $\hat{\mathbf{G}}_{\mathrm{rad}(\mathbf{G}, \rho)}$ along the minimal radius.

The Vakil--Zinger mapping space $\widetilde{\mathcal{M}}_{1,n}(\mathbb{P}^r,d)$ is cut out in the fiber product $\Mbar^{\mathrm{rad}}_{1,n}(\mathbb{P}^r,d)$ by the \textit{factorization property}. To define this, let \begin{equation}\label{eq:radial_aligned_map}(f: (C, p_1, \ldots, p_n) \to \mathbb{P}^r,\rho, L_1, \ldots, L_k)
\end{equation} be a radially-aligned stable map, i.e., a point in the fiber product $\Mbar^{\rad}_{1,n}(\mathbb{P}^r,d)$. The contraction radius $\mathrm{rad}(\mathbf{G}, \rho)$ together with the alignment data induces a pair of maps $$C\longleftarrow \widetilde{C}_{\mathrm{rad}(\mathbf{G}, \rho)}\xrightarrow{\eta_{\rad(\mathbf{G, \rho)}}} \overline{C}_{\mathrm{rad}(\mathbf{G}, \rho)},$$ where $\widetilde{C}_{\mathrm{rad}(\mathbf{G}, \rho)}\to C$ is the logarithmic modification that is induced by the subdivision $\hat{\mathbf{G}}_{\mathrm{rad}(\mathbf{G}, \rho)}$, and $\widetilde{C}_{\mathrm{rad}(\mathbf{G}, \rho)}\to \overline{C}_{\mathrm{rad}(\mathbf{G}, \rho)}$ is a contraction to an elliptic singularity.
\begin{defn}\label{defn-facprop}
The space $\widetilde{\mathcal{M}}_{1,n}(\mathbb{P}^r,d)$ is the locus in $\Mbar^{\rad}_{1,n}(\mathbb{P}^r,d)$ where the composition \[\widetilde{C}_{\mathrm{rad}(\mathbf{G}, \rho)}\to \mathbb{P}^r\] factors through $\overline{C}_{\mathrm{rad}(\mathbf{G}, \rho)}\to \mathbb{P}^r$. We call this condition the \textit{factorization property}. We equip the space $\widetilde{\mathcal{M}}_{1,n}(\mathbb{P}^r,d)$ with the pullback stratification from $\Mbar^{\rad}_{1,n}(\mathbb{P}^r,d)_{(\mathbf{G},\rho)}$, i.e., the strata are $$\widetilde{\M}(\mathbf{G}, \rho) := \Mbar^{\rad}_{1,n}(\mathbb{P}^r,d)_{(\mathbf{G},\rho)}\cap \widetilde{\mathcal{M}}_{1,n}(\mathbb{P}^r,d).$$
\end{defn}

A radially-aligned map $f$ as in (\ref{eq:radial_aligned_map}) satisfies the factorization property if and only if the kernel of the linear map
\[ \bigoplus_{v \in \rho^{-1}(\rad(\mathbf{G}, \rho))} T_{{\mu}_v} C_v \to T_p(\mathbb{P}^r) \]
induced by the derivative of $f$ contains the line $L_{\rad(\mathbf{G}, \rho)}$. See \cite[\S 2.4]{bnr} for a detailed discussion of how this relates to elliptic singularities.

Pleasing geometric properties of $\widetilde{\mathcal{M}}_{1,n}(\mathbb{P}^r,d)$ are witnessed by logarithmic deformation theory:
\begin{thm}\label{thm-logsm}
    The strata $\widetilde{\M}(\mathbf{G}, \rho)$ form the underlying stratification of a logarithmically smooth log structure on $\widetilde{\mathcal{M}}_{1,n}(\mathbb{P}^r,d)$. In particular, each stratum $\widetilde{\M}(\mathbf{G}, \rho)$ is smooth, and the compactification $\mathcal{M}_{1,n}(\mathbb{P}^r,d)\subset \widetilde{\mathcal{M}}_{1,n}(\mathbb{P}^r,d)$ has normal crossings boundary. 
\end{thm}
\begin{proof}
    This statement is implicit in \cite[\S4.5]{rspw}, but we include the following proof for completeness. To begin with, the stack of prestable radially-aligned curves $\mathfrak{M}_{1,n}^{\rad}$ is constructed as a logarithmic modification of $\mathfrak{M}_{1,n}$ \cite[Prop. 3.3.4]{rspw}: indeed, as explained in Remark \ref{pathlength}, the logarithmic modification is induced by a subdivision of polyhedral cone complexes and is locally given by toric blow-ups. Therefore, as $\mathfrak{M}_{1,n}$ is logarithmically smooth with normal crossings boundary, so is $\mathfrak{M}_{1,n}^{\rad}.$ 
    
    The universal Picard stack $\mathfrak{Pic}^{d}_{\mathfrak{M}_{1,n}^{\rad}}$ is smooth over $\mathfrak{M}_{1,n}^{\rad},$ because it is pulled back from the universal Picard stack on $\Mfr_{1,n}.$ Let $\mathfrak{Pic}^{d, \geq 0}_{\mathfrak{M}_{1,n}^{\rad}}$ be the open substack of $\mathfrak{Pic}^{d}_{\mathfrak{M}_{1,n}^{\rad}}$ parametrizing line bundles that have non-negative degree on each component, and let $\pi: \mathfrak{Pic}^{d, \geq 0}_{\mathfrak{M}_{1,n}^{\rad}}\to \mathfrak{M}_{1,n}^{\rad}$ be the forgetful map. For a prestable, radially-aligned $(1,n,0)$-graph $(\mathbf{G}_0,\rho),$ denote the locally closed substack associated to $(\mathbf{G}_0,\rho)$ as $\mathfrak{M}^{\rad}(\mathbf{G}_0,\rho)\subset \mathfrak{M}_{1,n}^{\rad},$ then $\pi^{-1}(\mathfrak{M}^{\rad}(\mathbf{G}_0,\rho))$ has connected components indexed by degree decorations $\delta: V(\mathbf{G}_0)\to \mathbb{Z}_{\geq 0}$ on $(\mathbf{G}_0,\rho)$ with total degree $d$; in other words, the connected components are precisely the radially-aligned $(1,n,d)$-graphs. In summary,  $\mathfrak{Pic}^{d, \geq 0}_{\mathfrak{M}_{1,n}^{\rad}}$ is smooth and logarithmically smooth with normal crossing boundary that is stratified by radially-aligned $(1,n,d)$-graphs.
    
    Finally, \cite[Remark 4.5.3]{rspw} states that the map from $\widetilde{\mathcal{M}}_{1,n}(\mathbb{P}^r,d)$ to the universal Picard stack over $\Mfr_{1,n}^{\mathrm{rad}}$ is smooth, and this map clearly factors through the open substack $\mathfrak{Pic}^{d, \geq 0}_{\mathfrak{M}_{1,n}^{\rad}}.$ Therefore, $\widetilde{\mathcal{M}}_{1,n}(\mathbb{P}^r,d)$ has the same properties as the ones of $\mathfrak{Pic}^{d, \geq 0}_{\mathfrak{M}_{1,n}^{\rad}},$ as desired. The smoothness of strata in $\widetilde{\mathcal{M}}_{1,n}(\mathbb{P}^r,d)$ follows from general facts about logarithmically smooth spaces.
\end{proof}

\section{Geometry of graph strata}\label{sec:graph_strata}

After reviewing the relevant moduli spaces in the previous section, we use the combinatorial gadgets in Section \ref{sec:combinatorics} to give a more explicit description of the strata of $\widetilde{\M}_{1,n}(\mathbb{P}^r,d)$. 

\subsection{Graph strata of $\Mfr_{1, n}^\rad$} We can stratify $\Mfr_{1, n}^\rad$ by radially-aligned dual graphs:

\begin{defn}
The \textit{combinatorial type} of a point in $\Mfr_{1, n}^{\rad}$ is the pair $(\mathbf{G}, \rho)$ consisting of the dual graph of the curve (which is a $(1, n, 0)$-graph) and the radial alignment $\rho$. Recall that we have the following combinatorial data associated to a combinatorial type:
\begin{enumerate}[(1)] 
\item If the codomain of $\rho$ is equal to $\{0, \ldots, k\}$ for $k \geq 0$,  we refer to $k$ as the \textit{length} of the radial alignment, and we write $\ell(\rho)$ for the length;
\item We write $C(\mathbf{G})$ for the set of edges of $\mathbf{G}$ contained in the core, and we write $T(\mathbf{G}) = E(\mathbf{G})\smallsetminus C(\mathbf{G})$ for the set of edges of $\mathbf{G}$ outside the core.
\item The set of vertices in the core is denoted by $V^{\mathrm{core}}(\mathbf{G})$ and its complement by $V^{\mathrm{tree}}(\mathbf{G})$.
\end{enumerate}
\end{defn}
As explained in Section \ref{subsec-radial}, we can think of a radial alignment as the data of an ordered partition of $T(\mathbf{G})$. Given a combinatorial type $(\mathbf{G}, \rho)$, let
\[\Mfr^\rad(\mathbf{G}, \rho) \subset \Mfr_{1,n}^\rad \]
denote the locally closed stratum of curves with combinatorial type equal to $(\mathbf{G}, \rho)$. This stratum maps to the stratum $\Mfr(\mathbf{G}) \subset \Mfr_{1, n}$ of prestable curves with dual graph equal to $\mathbf{G}$.

We now give a modular interpretation of the strata blow-up $\Mfr^{\rad}_{1,n}\to \Mfr_{1,n}$ from Definition \ref{defn-radps}:
\begin{lem}\label{lem-curvetorus}
Suppose $\rho$ is a radial alignment of $\mathbf{G}$. Then the map $\Mfr^\rad(\mathbf{G}, \rho) \to \Mfr(\mathbf{G})$ is a torsor with structure group
$$\mathbb{G}_m^{|V^{\mathrm{tree}}(\mathbf{G})| - \ell(\rho)} = \prod_{i=1}^{\ell(\rho)}\mathbb{G}_m^{|\rho^{-1}(i)| - 1}.$$ Furthermore, each factor $\mathbb{G}_m^{|\rho^{-1}(i)| - 1}$ of the structure group has the following two equivalent characterizations: \begin{itemize}\item Collections of isomorphisms $$\{T_{{\mu}_{v}}C_v\cong T_{{\mu}_{w}}C_w\}_{v,w\in \rho^{-1}(i)}$$ that are compatible under compositions.
\item Generic lines in the direct sum $\bigoplus_{v\in \rho^{-1}(i)}T_{{\mu}_v}C_v$, i.e., the dense torus in $\mathbb{P}\left(\bigoplus_{v\in \rho^{-1}(i)}T_{{\mu}_v}C_v\right)$.

\end{itemize}
\end{lem}
\begin{proof}Recall from Remark \ref{pathlength} that the strata blow-up is pulled back from the morphism of Artin fans $\mathcal{A}_{\Gamma^{\rad, \mathrm{ps}}_{1,n}(0)}\to \mathcal{A}_{\Gamma^{\mathrm{ps}}_{1,n}(0)}$, which in turn is induced by subdividing along $\{\mathsf{dis}_x = \mathsf{dis}_y\}_{x,y\in V^{\mathrm{tree}}(\mathbf{G})}$ on the level of cones. 

The following discussion expands \cite[Lemma 3.3.2]{rspw} and explains how the distance functions are related to line bundles on the prestable curves. Via the correspondence between piecewise linear functions and (toric) line bundles, the two functions $\mathsf{dis}_x$ and $\mathsf{dis}_y$ produce line bundles on $\mathcal{A}_{\Gamma^{\mathrm{ps}}_{1,n}(0)}$ that we denote as $\mathcal{L}(P_x)$ and $\mathcal{L}(P_y)$ respectively. 

Let $BT^{(\mathbf{G},\rho)}\subset \mathcal{A}_{(\mathbf{G}, \rho)}$ and $BT^{\mathbf{G}}\subset \mathcal{A}_{\mathbf{G}}$ denote the minimal closed substacks of the Artin cones $\mathcal{A}_{(\mathbf{G}, \rho)}$ and $ \mathcal{A}_{\mathbf{G}}$ respectively: they are classifying stacks\footnote{Working with étale local toric charts, they are quotient stacks of the torus fixed points on the toric varieties.} of algebraic tori that have dimensions equal to the dimensions of the cones associated to $(\mathbf{G}, \rho)$ and $\mathbf{G},$ respectively, which we denote as $T^{(\mathbf{G},\rho)}$ resp. $T^{\mathbf{G}}$. Note that because the cone of $(\mathbf{G},\rho)$ comes from subdividing that of $\mathbf{G},$ we have $\dim T^{(\mathbf{G},\rho)}\leq \dim T^{\mathbf{G}}.$ The map $\mathcal{A}_{(\mathbf{G}, \rho)}\to  \mathcal{A}_{\mathbf{G}}$ restricts to a map $BT^{(\mathbf{G},\rho)}\to BT^{\mathbf{G}}.$ Via the general dictionary for passing between subdivisions, logarithmic modifications, and toric blow-ups, the fiber of the map $BT^{(\mathbf{G},\rho)}\to BT^{\mathbf{G}}$ is the algebraic torus that parametrizes precisely a collection of compatible isomorphisms of line bundles $\{\mathcal{L}(P_x)\cong\mathcal{L}(P_y)\}_{x,y\in \rho^{-1}(i)}$ for each $i>0$.

Pulling back to the strata blow-ups, the compatible isomorphisms become that of $$\{\mathcal{O}_{\mathfrak{M}(\mathbf{G})}(P_x)\cong \mathcal{O}_{\mathfrak{M}(\mathbf{G})}(P_y)\}_{x,y\in \rho^{-1}(i)}$$ for each $i>0$. 

We unravel the isomorphisms of the line bundles by recalling the following general fact: see \cite[§5.5]{MPS} for more detail. Let $\mathbf{G}$ be a prestable $(1,n,0)$-graph, $e\in E(\mathbf{G}),$ and let $\mathbf{G}/e$ be the edge contraction. Then ${\Mfr(\mathbf{G})}$ is the interior of a divisor in the stratum closure $\overline{\Mfr(\mathbf{G}/e)}.$ On the other hand, the edge length of $e$ defines a piecewise linear function on the polyhedral cone $\mathbb{R}_{\geq 0}^{|E(\mathbf{G})|}/\mathrm{Aut}(\mathbf{G})$ in the cone stack $\Sigma_{\Gamma_{1,n}^{\mathrm{ps}}(0)},$ and the corresponding line bundle pulls back to the normal bundle of ${\Mfr(\mathbf{G})}$ in $\overline{\Mfr(\mathbf{G}/e)}$: more precisely, we take the normal bundle of the divisor and then restrict it to ${\Mfr(\mathbf{G})}.$ By the definition given in Remark \ref{pathlength}, the piecewise linear functions $\mathsf{dis}_x$ and $\mathsf{dis}_y$ are the lengths of the unique paths $P_x$ and $P_y$ that connect the vertices $x$ and $y$ to the core. Therefore, $\mathcal{L}(P_x)$ and $\mathcal{L}(P_y)$ pull back to the tensor product of the normal bundles, one for each edge in $P_x$ resp. $P_y.$ 

Lemma 4.15 of \cite{bnr} states that $\mathcal{L}(P_x)$ resp. $\mathcal{L}(P_y)$ can be identified with the following line bundles on ${\Mfr(\mathbf{G})}$ that we now construct. Let $e_x$ resp. $\tilde{e}_x$ be the first and last edge of the unique minimal path that connects $x$ to the core. Similarly define edges $e_y$ and $\tilde{e}_y.$ 

Let $\mathbf{G}_{\geq x}$ be the subgraph consisting of the vertices that are greater or equal to $x$ with respect to the partial order, together with one leg $\ell_x^{\geq}$ added that corresponds to `cutting' the edge $e_x.$ Similar for $\mathbf{G}_{\geq y}$ and $\ell_y^{\geq}.$ On the other hand, let $\mathbf{G}_{\smallsetminus \tilde{e}_x}$ be the subgraph that replaces the path leading to and including $\mathbf{G}_{\geq x}$ by a leg $\ell_x^{\leq}$ that corresponds to `cutting' the edge $\tilde{e}_x.$ Similar for $\mathbf{G}_{\smallsetminus \tilde{e}_y}$ and $\ell_y^{\leq}.$

Cutting the edges $e_x,\tilde{e}_x$ and $e_y,\tilde{e}_y$ gives forgetful morphisms $$\mathfrak{M}(\mathbf{G})\to \mathfrak{M}(\mathbf{G}_{\geq x})\times \mathfrak{M}(\mathbf{G}_{\smallsetminus \tilde{e}_x}),$$ $$\mathfrak{M}(\mathbf{G})\to \mathfrak{M}(\mathbf{G}_{\geq y})\times \mathfrak{M}(\mathbf{G}_{\smallsetminus \tilde{e}_y}).$$ We denote the universal curves over the above moduli stacks of curves as $\mathfrak{C}\to \mathfrak{M}.$ Observe that the distinguished legs $\ell$ in all four occasions correspond to marked points that give sections $\ell: \mathfrak{M}\to \mathfrak{C}.$ Let $N_{\ell}$ be the line bundles on $\mathfrak{M}$ associated to the normal bundles of $\ell: \mathfrak{M}\to \mathfrak{C}.$ Then \cite[Lemma 4.15]{bnr} states that $$\mathcal{L}(P_x)\cong N_{\ell_x^{\leq}}\otimes N_{\ell_x^{\geq}}, \mathcal{L}(P_y)\cong N_{\ell_y^{\leq}}\otimes N_{\ell_y^{\geq}}.$$ Let $C\in \Mfr(\mathbf{G}),$ let $(C_{\smallsetminus \tilde{e}_x}, \mu_{x}^{\leq})$ be the marked subcurve of $C$ corresponding to $(\mathbf{G}_{\smallsetminus \tilde{e}_x}, \ell_{x}^{\leq}),$ and let $(C_x, \mu_x)$ be as before, which corresponds to $(\mathbf{G}_{\geq x}, \ell_{x}^{\geq}).$ Then the fiber of $N_{\ell_x^{\leq}}$ at $C$ corresponds to $T_{\mu_{x}^{\leq}}C_{\smallsetminus \tilde{e}_x},$ and the fiber of $N_{\ell_x^{\geq}}$ at $C$ corresponds to $T_{\mu_{x}}C_{x}.$ As explained after the statement of \cite[Lemma 4.15]{bnr}, the fact that we work over genus one prestable curves provides canonical isomorphisms $N_{\ell_x^{\leq}}\cong N_{\ell_{y}^{\leq}},$ and this line bundle is denoted as $\mathbb{T}$ by the authors in \cite{bnr}. Therefore, the isomorphisms $\mathcal{L}(P_x)\cong \mathcal{L}(P_y)$ are equivalent to isomorphisms $N_{\ell_x^{\geq}}\cong N_{\ell_y^{\geq}}:$ again, over each point $C\in \Mfr(\mathbf{G}),$ this is the data of an isomorphism $T_{\mu_x}C_x\cong T_{\mu_y}C_y.$

Finally, as explained in \cite[§2]{bnr} (after Corollary 2.2), a collection of compatible isomorphisms among $\{T_{{\mu}_{v}}C_v\}_{v\in \rho^{-1}(i)}$ is identified with a point in the dense torus in $\mathbb{P}\left(\bigoplus_{v\in \rho^{-1}(i)}T_{{\mu}_v}C_v\right)$ by taking isomorphisms $$\left(\theta_{1j}: T_{{\mu}_{v_1}}C_{v_1}\xrightarrow{\cong}T_{{\mu}_{v_j}}C_{v_j}\right)_{v_j\in \rho^{-1}(i)\smallsetminus\{v_1\}}$$ to the image line $$(\mathrm{id}, \theta_{ij}):T_{{\mu}_{v_1}}C_{v_1}\to \bigoplus_{v\in \rho^{-1}(i)}T_{{\mu}_{v}}C_{v}.$$
\end{proof}

\begin{cor}
    The codimension of $\Mfr^\rad(\mathbf{G}, \rho)$ in $\Mfr_{1, n}^\rad$ is \[|E(\mathbf{G})| - |V^\mathrm{tree}(\mathbf{G})| + \ell(\rho) =  |C(\mathbf{G})| + \ell(\rho).\]
\end{cor}
\begin{proof}
    From the previous Lemma, the map $\Mfr^\rad(\mathbf{G}, \rho)\to \Mfr(\mathbf{G})$ has fibers with dimension $|V^\mathrm{tree}(\mathbf{G})|-\ell(p),$ hence $\dim \Mfr^\rad(\mathbf{G}, \rho) = \dim \Mfr(\mathbf{G}) + |V^\mathrm{tree}(\mathbf{G})|-\ell(p).$ On the other hand, because $\Mfr^{\rad}_{1,n}\to \Mfr_{1,n}$ is birational, $\dim \Mfr^{\rad}_{1,n} = \Mfr_{1,n}.$ Therefore, \begin{align*}
        \mathrm{codim}(\Mfr^\rad(\mathbf{G}, \rho)\subset \Mfr^{\rad}_{1,n})& =\dim \Mfr^{\rad}_{1,n} - \dim \Mfr^\rad(\mathbf{G}, \rho)\\ &=  \dim  \Mfr_{1,n} - (\dim \Mfr(\mathbf{G}) + |V^\mathrm{tree}(\mathbf{G})|-\ell(p)) \\ & = \mathrm{codim}(\mathfrak{M}(\mathbf{G})\subset \mathfrak{M}_{1,n}) + |V^\mathrm{tree}(\mathbf{G})|-\ell(p) \\& = |E(\mathbf{G})| - |T(\mathbf{G})| + \ell(p),
    \end{align*}
    where the last equality uses the fact that $\mathrm{codim}(\mathfrak{M}(\mathbf{G})\subset \mathfrak{M}_{1,n}) = |E(\mathbf{G})|.$ Finally, from Remark \ref{rem:vtreeT}, the bijection between $V^{\mathrm{tree}}(\mathbf{G})$ and $T(\mathbf{G})$ implies $|V^{\mathrm{tree}}(\mathbf{G})| = |T(\mathbf{G})|,$ so that $|E(\mathbf{G})| - |V^\mathrm{tree}(\mathbf{G})| = |C(\mathbf{G})|$ as claimed.
\end{proof}
\subsection{Strata of mapping spaces}\label{map}
Let $(\mathbf{G},\rho)\in \Gamma^{\mathrm{rad}}_{1,n}(d)$ be a radially-aligned stable $(1,n,d)$-graph. It specifies a stratum $\mathcal{M}(\mathbf{G},\rho)\subset \Mbar^{\rad}_{1,n}(\mathbb{P}^r,d)$.
\begin{lem}\label{lem:mapfibbun}
The restriction of the strata blow-up $\Mbar^{\rad}_{1,n}(\mathbb{P}^r,d)\to \Mbar_{1,n}(\mathbb{P}^r,d)$ to the stratum $\mathcal{M}(\mathbf{G},\rho)$ is a $\mathbb{G}_m^{|V^{tree}(\mathbf{G})|-\ell(\rho)}$-fiber bundle.
\end{lem}
\begin{proof}
Recall from Lemma \ref{lem-curvetorus} that the map on the moduli of prestable curves $\Mfr(\mathbf{G},\rho)\to \Mfr(\mathbf{G})$ is a $\mathbb{G}_m^{|T(\mathbf{G})|-\ell(\rho)}$-fiber bundle. The desired statement follows from taking the fiber product \Cartesiansquare{\mathcal{M}(\mathbf{G},\rho)}{\mathcal{K}(\mathbf{G})}{\Mfr(\mathbf{G},\rho)}{\Mfr(\mathbf{G}).}
\end{proof}

Recall that \[\widetilde{\M}(\mathbf{G},\rho) = \M(\mathbf{G}, \rho) \cap \widetilde{\M}_{1, n}(\mathbb{P}^r, d) \subseteq \widetilde{\M}_{1, n}(\mathbb{P}^r, d).\] 

    We conclude this section by noting how these graph strata fit together.
    \begin{thm}\label{thm: PosetStructure}
        We have a containment \[\widetilde{\M}(\mathbf{G}, \rho) \subseteq \overline{\widetilde{\M}(\mathbf{G}', \rho')}\] if and only if there is a morphism $(\mathbf{G}, \rho) \to (\mathbf{G}', \rho')$ in $\tgamma^\rad_{1, n}(d)$.
    \end{thm}
    \begin{proof}
        The statement is equivalent to showing that the stratum $\widetilde{\M}(\mathbf{G},\rho)$ admits a smoothing to $\widetilde{\M}(\mathbf{G}',\rho')$ if and only if there is a morphism $(\mathbf{G}, \rho) \to (\mathbf{G}', \rho')$ in $\tgamma^\rad_{1, n}(d)$.

        We observe that the analogous statement holds for the stack of radially-aligned prestable curves. As the stack $\Mfr_{1,n}^{\rad}\to \Mfr_{1,n}$ is a strata blow-up, it is logarithmically smooth. Therefore, the containments of strata in $\Mfr_{1,n}^{\rad}$ correspond to morphisms in the category $\tilde{\Gamma}^{\mathrm{ps}}_{1,n}(0)$. 

        The morphism $\widetilde{\M}_{1,n}(\mathbb{P}^r,d)\to \Mfr_{1,n}^{\rad}$ is smooth from \cite[Theorem 4.5.1]{rspw}. The smoothness implies that given a smoothing of the prestable, radially-aligned domain curves, we may always lift this to a smoothing of the maps. Hence $\widetilde{\M}(\mathbf{G},\rho)$ smoothes to $\widetilde{\M}(\mathbf{G}',\rho')$ if and only there is some lift of a morphism between the two underlying radially-aligned graphs. As a smoothing preserves the degree of the map (indeed the Hilbert polynomial), such a lift has the degree assignments determined by the underlying morphism of aligned graphs and must come from a morphism in $\tilde{\Gamma}_{1,n}^\rad(d)$, as desired.
    \end{proof}
\section{Irreducibility of strata}\label{sec:connectedness}
The goal of this section is to prove the first part of Theorem \ref{thm:determination}, which states that the locally closed strata $\widetilde{\M}(\mathbf{G},\rho)$ of $\widetilde{\M}_{1,n}(\mathbb{P}^r,d)$ are irreducible. We also give a criterion for when such a stratum is non-empty. These results are summarized in Theorem \ref{conprop} below. Both are crucial for determining the dual complex of $\M_{1,n}(\mathbb{P}^r,d)\subset \widetilde{\M}_{1,n}(\mathbb{P}^r,d)$. 

For $(\mathbf{G},\rho)$ in which the core\footnote{Abusing notation, we use the term `core' to denote both the genus one subgraph of $\mathbf{G}$ and the unique minimal genus one subcurve of the prestable curve corresponding to the core.} is not contracted, the factorization property is trivial. Therefore $\widetilde{\M}(\mathbf{G},\rho)\to \K(\mathbf{G})\subset \Mbar_{1,n}(\mathbb{P}^r,d)$ is the total space of a torus bundle over the stable map stratum $\K(\mathbf{G})$. In Section \ref{sec:strataK}, we prove that $\K(\mathbf{G})$ is irreducible when $g \leq 1$. Therefore, when the core has positive degree, the stratum $\widetilde{\M}(\mathbf{G},\rho),$ is irreducible as well. Care is needed only when the genus one core is contracted. In this situation, the radial alignment $\rho$ gives a non-zero linear dependency condition on the tangent vectors.

We begin by discussing the space of \textit{parameterized} maps from $\P^1$ to $\P^r$, as well as auxiliary spaces constructed by imposing tangent vector conditions on such maps.  


\subsection{Parameterized maps with a fixed tangent vector}

\begin{defn}
    Let $$\mathrm{Map}_d^*(\mathbb{P}^1, \mathbb{P}^r):=\{ \mathbb{P}^1\xrightarrow{f} \mathbb{P}^r \mid [1:0]\mapsto [1:1:\cdots:1], f^*\mathcal{O}_{\mathbb{P}^r}(1)\cong \mathcal{O}_{\mathbb{P}^1}(d)\}$$ be the space of pointed degree $d$ maps from $\mathbb{P}^1$ to $\mathbb{P}^r$.
\end{defn}
\begin{rem}
    While we have chosen the distinguished points on the domain and target to be $[1:0]\in \mathbb{P}^1$ and $[1:1:\cdots:1]\in \mathbb{P}^r$ respectively, the space $\mathrm{Map}_d^*(\mathbb{P}^1, \mathbb{P}^r)$ does not depend on the choices. 
    
    We also note that $\mathrm{Map}_d^*(\mathbb{P}^1, \mathbb{P}^r)$ is a closed subvariety of the space of maps $\mathrm{Map}_d(\mathbb{P}^1,\mathbb{P}^r),$ which is isomorphic to $\mathcal{M}_{0,3}(\mathbb{P}^r,d)$ by identifying the three ordered marked points parametrized by $\mathcal{M}_{0,3}(\mathbb{P}^r,d)$ with $0,1,\infty\in \mathbb{P}^1.$
\end{rem}
We start by identifying the space of pointed maps from $\mathbb{P}^1$ to $\mathbb{P}^r$  with a space of polynomials \cite[Definition 1.1]{farbwolf}. Namely, the space $\mathrm{Map}_d^*(\mathbb{P}^1, \mathbb{P}^r)$ is isomorphic to $$\{(f_0,\dots,f_r)\mid f_i\in \mathbb{C}[z]\text{ monic of degree }d, \text{no common factor among }f_i\}\subset \mathbb{A}_{\mathbb{C}}^{d(r+1)},$$which is a dense open subset in the affine space. Specifically, the polynomials specify the map on the affine chart $[z:1]\mapsto [f_0(z):\dots:f_r(z)]$. It uniquely extends to the other chart by$$[1:w]\mapsto [w^d f_0(1/w):\dots: w^d f_r(1/w)]$$ and in particular $[1:0]\mapsto [1:1:\dots:1]$.

Since the polynomials $f_i$ are monic, they are uniquely determined by their sets of roots, hence the space $\mathrm{Map}_d^*(\mathbb{P}^1, \mathbb{P}^r)$ may also be described as $$\{(R_0,\dots,R_r)\mid R_i\in \mathrm{Sym}^d(\mathbb{A}_{\mathbb{C}}^1), \bigcap_{i=0}^r R_i=\varnothing\}.$$ The complement is$$\mathsf{B}:=\mathbb{A}^{d(r+1)}_{\mathbb{C}}\smallsetminus \mathrm{Map}_d^*(\mathbb{P}^1, \mathbb{P}^r)=\{(R_i)_{i=0}^r\mid R_i\in \mathrm{Sym}^d(\mathbb{A}^1_{\mathbb{C}}), \bigcap_{i=0}^r R_i\neq \varnothing\},$$ and an element in $\bigcap_{i=0}^r R_i$ is called a `basepoint,' as these are the points on $\mathbb{A}_{\mathbb{C}}^1$ that prevent $(f_i)_{i=0}^r$ from defining a degree-$d$ map $\mathbb{P}^1\to \mathbb{P}^r$.

\begin{rem}
    The definition of $\mathrm{Map}_d^*(\mathbb{P}^1,\mathbb{P}^r)$ makes sense when $d = 0.$ In this case, $\mathrm{Map}_0^*(\mathbb{P}^1,\mathbb{P}^r)$ is a single point, which is the constant map taking $\mathbb{P}^1$ to $[1:1:\cdots:1]$ on the target. The monic polynomials are $f_0,\dots, f_r\equiv 1,$ with their sets of roots as the empty set.
\end{rem}

Let $p=[1:\dots:1]\in \mathbb{P}^r$. We now describe the locus of parameterized maps with a fixed tangent vector at the marked point $[1:0]$. \begin{lem}\label{lem:partial10}
    \label{lem}The map $\partial_{[1:0]}:\mathrm{Map}_d^*(\mathbb{P}^1,\mathbb{P}^r)\to T_p\mathbb{P}^r$ of taking derivatives at $[1:0]\in \mathbb{P}^1$ is given by \[(R_0,\dots,R_r)\mapsto \left[\left(-\sum_{y\in R_0}y,\dots,-\sum_{y\in R_r}y\right)\right]\in \mathbb{C}^{r+1}/\mathbb{C}\cdot (1,\dots,1)=T_{p}\mathbb{P}^r.\]The map is smooth onto its image and has non-empty and irreducible fibers over $\mathbf{v}\in T_{p}\mathbb{P}^r$ unless $d=1$ and $\mathbf{v}=0\in T_p\mathbb{P}^r$.
\end{lem}
\begin{proof}
We use the description of $\mathrm{Map}^*_d(\mathbb{P}^1,\mathbb{P}^r)$ as an $(r+1)$-tuple of polynomials and denote the coefficients of the polynomials as $$f_i(z) = z^d + f_{i, d-1}z^{d-1}+\dots +f_{i, 0},$$ which are extended to the chart $\{[1:w]\}$ by $$\tilde{f}_i(w) = 1 + f_{i, d-1}w + \dots + f_{i,0}w^d.$$ The derivative $$\left(\frac{\partial}{\partial w}\tilde{f}_i\right)_{w=0}=f_{i,d-1}=-\sum_{y\in R_i}y$$ leads to the expression of $\partial_{[1:0]}$ claimed above.

The fiber $\partial_{[1:0]}^{-1}(\mathbf{v})$ is the intersection of $\mathrm{Map}^*_d(\mathbb{P}^1,\mathbb{P}^r)\subset \mathbb{A}^{d(r+1)}$ and the affine\footnote{The term `affine' here is meant in the sense of linear algebra: an affine subspace is the translation of a linear subspace in an ambient vector space.} linear subspace $$\mathsf{L}_{\mathbf{v}}:=\{(f_i)_{i=0}^{r}\in \mathbb{A}^{d(r+1)}|(f_{0, d-1},\dots, f_{r,d-1})\in \mathbf{v}+\langle(1,\dots,1)\rangle\}.$$ Therefore, if $\partial_{[1:0]}^{-1}(\mathbf{v})$ is not empty, it is a dense open subset of the affine subspace $\mathsf{L}_{\mathbf{v}}$, which is smooth and irreducible.  

We now prove that the fiber is non-empty in the cases described above. The intersection is empty if and only if $\mathsf{L}_{\mathbf{v}}\subset \mathsf{B}=\mathbb{A}^{d(r+1)}\smallsetminus \mathrm{Map}^*_d(\mathbb{P}^1,\mathbb{P}^r)$. When $d=1$, the condition $\mathbf{v}\neq 0$ implies that the $r + 1$ sets $R_i$ that correspond to tuples of polynomials $\mathsf{L}_{\mathbf{v}}$ are not identical, so there is no basepoint. In other words, $\mathsf{L}_{\mathbf{v}}\cap \mathsf{B}=\varnothing$ in this case, so in particular $\mathsf{L}_{\mathbf{v}}\not\subset \mathsf{B}$. On the other hand, $d=1,\mathbf{v}= 0$ forces the sets $R_i$ correspond to elements in $\mathsf{L}_{\mathbf{0}}$ to be identical, hence $\mathsf{L}_{\mathbf{0}}\subset \mathsf{B}$.

Suppose $(R_0,\dots, R_r)$ corresponds to some $(f_0,\dots,f_r)\in \mathsf{L}_{\mathbf{v}}\cap \mathsf{B}$, so that $\bigcap_{i=1}^r R_i=\{z_1,\dots, z_k\}$, where $1\leq k\leq d$. When $d\geq 2$, so that $|R_0|\geq 2$, observe that there always exists some $R_0'$ such that $\sum_{y\in R_0}y = \sum_{y'\in R_0'}y'$ and that $R_0'\cap \{z_1,\dots, z_k\}=\varnothing$. Thus, $\mathsf{L}_{\mathbf{v}}\not\subset \mathsf{B}$ in this case.

Therefore, we have shown that when $d = 1,$ the map $\partial_{[1:0]}:\mathrm{Map}^*_1(\mathbb{P}^1,\mathbb{P}^r)\to T_p\P^r\smallsetminus 0$ is surjective with smooth fibers. Since the domain and target are smooth, it follows that the map $\partial_{[1:0]}$ is smooth when $d=1.$ When $d\geq 2,$ the same reasoning implies that the map $\partial_{[1:0]}:\mathrm{Map}^*_1(\mathbb{P}^1,\mathbb{P}^r)\to T_p\P^r$ is smooth.
\end{proof}
\begin{rem}
    The exception of $d=1$, $\mathbf{v}=0$ has a clear geometric picture: the degree one maps $\mathbb{P}^1\to \mathbb{P}^r$ are linear embeddings  $\mathbb{P}^1\subset \mathbb{P}^r$ and hence cannot have vanishing tangent vectors.
\end{rem}
Recall that $\mathsf{r} := \mathrm{rad}(\mathbf{G}, \rho)$ denotes the contraction radius.
\begin{defn}
    Let $\widetilde{\mathrm{Map}}_{\mathbf{G}}^{*, \mathsf{F}, \rho}$ be the locus in $\prod_{v\in \rho^{-1}(\mathsf{r})}\mathrm{Map}^*_{\delta(v)}(\mathbb{P}^1,\mathbb{P}^r)\times (\mathbb{C}^\star)^{|\rho^{-1}(\mathsf{r})|}/\mathbb{C}^\star$ that parametrize a collection of pointed maps together with a non-vanishing linear dependency (up to common rescaling) of the representative tangent vectors at the marked points. 
\end{defn}
\begin{lem}\label{mapcon}
     The space $\widetilde{\mathrm{Map}}_{\mathbf{G}}^{*, \mathsf{F}, \rho}$ is irreducible. Furthermore, it is empty if and only if $d_{\min}(\mathbf{G}, \rho) = 1$. 
\end{lem}
\begin{proof}
    Label the vertices of $\rho^{-1}(\mathsf{r})$ as $v_1, \ldots v_j$ where $j = |\rho^{-1}(\mathsf{r})|$. The derivative maps $\mathrm{Map}_d^*(\mathbb{P}^1,\mathbb{P}^r)\to T_p\mathbb{P}^r$ assemble to $$ \gamma: \prod_{i = 1}^{j}\mathrm{Map}^*_{\delta(v_i)}(\mathbb{P}^1,\mathbb{P}^r) \to \bigoplus_{i = 1}^{j} T_p \mathbb{P}^r.$$ We may take a product with the torus $(\mathbb{C}^\star)^{j}/\mathbb{C}^\star$ to get $$\gamma\times \mathrm{id}: \prod_{i = 1}^{j}\mathrm{Map}^*_{\delta(v_i)}(\mathbb{P}^1,\mathbb{P}^r)\times \left((\mathbb{C}^\star)^{j}/\mathbb{C}^\star\right) \to \bigoplus_{i = 1}^{j} T_p \mathbb{P}^r\times \left((\mathbb{C}^\star)^{j}/\mathbb{C}^\star\right).$$

     
     Write $\delta_i := \delta(v_i)$. The space of tangent vectors with linear dependencies that will appear in $\widetilde{\mathrm{Map}}_{\mathbf{G}}^{*, \mathsf{F}, \rho}$ is the subspace of $\bigoplus_{i = 1}^{j} T_p \mathbb{P}^r\times \left((\mathbb{C}^\star)^{j}/\mathbb{C}^\star\right)$ given by $${\mathsf{D}}_{\boldsymbol{\delta}}:=\{(\mathbf{v}_1,\dots, \mathbf{v}_j, [\alpha_1,\dots,\alpha_j])\mid \sum_{i=1}^j \alpha_i \mathbf{v}_i = 0, \mathbf{v}_i\neq 0\text{ when } \delta_i = 1, \mathbf{v}_i  = 0 \text{ when }\delta_i = 0\},$$ which is open in $$\mathsf{D}_{\boldsymbol{\delta}}':=\{(\mathbf{v}_1,\dots, \mathbf{v}_j, [\alpha_1,\dots,\alpha_j])\mid \sum_{i=1}^j \alpha_i \mathbf{v}_i = 0, \mathbf{v}_i  = 0 \text{ when }\delta_i = 0\}.$$

     We claim that $\mathsf{D}_{\boldsymbol{\delta}}'$ is irreducible. First, in the case where all entries in the degree vector $\boldsymbol{\delta}$ are zero, we have \[\mathsf{D}_{\boldsymbol{\delta}}' = \{(0,\dots, 0, [\alpha_1,\dots,\alpha_j])\}\cong (\C^\star)^j/\C^\star,\] which is irreducible. Otherwise, after reordering the indices, we may assume that $\delta_j\neq 0.$ There is an isomorphism \[\phi:\left(\bigoplus_{\substack{i\neq j \\ \delta_i\neq 0}}T_p\P^r\right)\times (\mathbb{C}^\star)^{j-1}\to \mathsf{D}_{\boldsymbol{\delta}}'\] given by $$((\mathbf{v}_i)_{{\substack{i\neq j \\ \delta_i\neq 0}}},(\alpha_i)_{i=1}^{j-1})\mapsto \left( \left((\mathbf{v}_i)_{{\substack{i: i\neq j \\ \delta_i\neq 0}}},(0)_{i':\delta_{i'}=0},-\sum_{{\substack{i: i\neq j \\ \delta_i\neq 0}}}\alpha_i \mathbf{v}_i\right), [\alpha_1,\dots, \alpha_{j-1},1]\right).$$

    
    Consider the open subset \[U_{\boldsymbol{\delta}}\subset \left(\bigoplus_{\substack{i\neq j \\ \delta_i\neq 0}}T_p\P^r\right)\times (\mathbb{C}^\star)^{j-1}\] given by the condition that $\mathbf{v}_i\neq 0$ whenever $\delta_i = 1.$ Since $U_{\boldsymbol{\delta}}$ is open in an irreducible space, we know that $U_{\boldsymbol{\delta}}$ is irreducible. Hence, $\mathsf{D}_{\boldsymbol{\delta}}',$ which is isomorphic to the image $\phi(U_{\boldsymbol{\delta}}),$ is irreducible. Further, $\mathsf{D}_{\boldsymbol{\delta}}$ is non-empty if and only if $U_{\boldsymbol{\delta}}$ is, which we can check is equivalent to the condition that $d_{\min}(\mathbf{G}, \rho) = 1.$

    The space of interest $\widetilde{\mathrm{Map}}_{\mathbf{G}}^{*, \mathsf{F}, \rho}$ is the preimage of $\mathsf{D}_{\boldsymbol{\delta}}$ under the map $\gamma\times \mathrm{id}$. Therefore, $\widetilde{\mathrm{Map}}_{\mathbf{G}}^{*, \mathsf{F}, \rho}$ is empty when $d_{\mathrm{min}}(\mathbf{G},\rho) = 1$ because $\mathsf{D}_{\boldsymbol{\delta}} = \varnothing.$
    
    We turn to the case of $d_{\mathrm{min}}(\mathbf{G},\rho)>1$ so that $\mathsf{D}_{\boldsymbol{\delta}}\neq \varnothing.$ Given a tuple $((\mathbf{v}_i)_{i=1}^j, [\alpha_1:\dots: \alpha_j])\in \mathsf{D}_{\boldsymbol{\delta}},$ the preimage $(\gamma \times \mathrm{id})^{-1}((\mathbf{v}_i)_{i=1}^j, [\alpha_1:\dots: \alpha_j])$ is canonically identified with $$\prod_{i=1}^j \partial_{[1:0]}^{-1}(\mathbf{v}_i)\subset \prod_{i=1}^j \mathrm{Map}^*_{\delta(v_i)}(\mathbb{P}^1,\mathbb{P}^r).$$ Lemma \ref{lem:partial10} gives that $\partial_{[1:0]}$ is smooth onto its image, and each $\partial_{[1:0]}^{-1}(\mathbf{v}_i)$ is non-empty and irreducible. The same then holds for the morphism $\gamma \times \mathrm{id} = (\prod_{i=1}^j \partial_{[1:0]}) \times \mathrm{id}$ and its fibers. Therefore, $\widetilde{\mathrm{Map}}_{\mathbf{G}}^{*, \mathsf{F}, \rho}$ is irreducible.
\end{proof}


We record another irreducibility result that will be relevant to the discussion of stable maps strata.
\begin{rem}\label{rem:two-pointed-maps}
    Let $q\in \mathbb{P}^r,$ and consider $\mathrm{Map}_d^{*,q}(\P^1, \P^r)$ be the space of degree $d$ parametrized maps $f: \P^1\to \P^r$ such that $f([1:0]) = [1:1:\cdots:1],$ and $f([0:1]) = q.$ This is the intersection of the affine linear subspace $$\{(f_0,\dots,f_r)\mid f_i\in\C[z]\text{ monic of degree }d\text{ and } (f_0(0),\dots,f_r(0))\in q\subset \mathbb{C}^{r+1}\}\subset \mathbb{A}_\mathbb{C}^{d(r+1)}$$ and the dense open $\mathrm{Map}_d^{*}(\P^1, \P^r)\subset \mathbb{A}_\mathbb{C}^{d(r+1)}$ and is thus irreducible. It is empty only when $d = 1$ and $q = [1:1:\cdots:1].$
\end{rem}
\subsection{Strata of the Kontsevich moduli space}\label{sec:strataK}
Recall that each $(g, n, d)$-graph $\mathbf{G}$ defines a locally closed stratum $\K(\mathbf{G})$ in the Kontsevich moduli space $\Mbar_{g, n}(\mathbb{P}^r, d)$. In this section, we will show that each $\K(\mathbf{G})$ is irreducible when $g \leq 1$. 



\begin{prop}\label{prop:conn_genus_zero}
Suppose $d > 0$. For any $(0, n, d)$-graph $\T$, the stratum $\K(\T)$ is irreducible.
\end{prop}
Our proof of Proposition \ref{prop:conn_genus_zero} will require the following lemma.
\begin{lem}[Proposition 4.1 of \cite{farbwolf}]
Suppose $n, d  > 0$, and let
\[ \mathrm{ev}_1 : \M_{0, n+1}(\P^r, d) \to \P^r\]
denote evaluation at the first marked point. Then the fiber
\[ \M_{0, n}^*(\P^r, d) \]
of $\mathrm{ev}_1$ over any point in $\P^r$ is irreducible. 
 \end{lem}
 \begin{proof}
We have an isomorphism $$\mathcal{M}_{0,n}^*(\P^r,d)\cong \left(\mathrm{Map}^*_{d}(\P^1,\P^r)\times \mathrm{Conf}^{n}(\mathbb{P}^1\smallsetminus [1:0])\right)/\mathrm{Aut}(\mathbb{P}^1, [1:0]),$$ where $\mathrm{Aut}(\mathbb{P}^1, [1:0])\subset \mathrm{PGL}_2$ is the subgroup of Moebius transformations that fixes the point $[1:0]\in \P^1.$ The irreducibility of $\M_{0, n}^*(\P^r, d)$ now follows from the fact that both $\mathrm{Map}^*_{d}(\P^1,\P^r)$ and $\mathrm{Conf}^{n}(\mathbb{P}^1\smallsetminus [1:0])$ are irreducible.
 \end{proof}

\begin{proof}[Proof of Proposition \ref{prop:conn_genus_zero}]
First, let $\K^{\mathrm{iso}}(\T)$ be the stack of stable maps with dual graph isomorphic to $\T$, together with a fixed isomorphism between the dual graph of the map and the graph $\T$. Then we have
\[\K(\T) \cong \K^{\mathrm{iso}}(\T) /\mathrm{Aut}(\T), \]
so it suffices to show that $\K^{\mathrm{iso}}(\T)$ is irreducible.

We proceed by induction on the number of vertices of $\T$. If $|V(\T)| = 1$, the desired statement is the well-known fact that $\M_{0, n}(\P^r, d)$ is irreducible. If $V(\T) > 1$, then we can find a leaf vertex (vertex of valence one) $v \in V(\T)$. Let $\T\smallsetminus v$ denote the $(0, n', d')$-tree obtained by deleting the vertex $v$ from $V(\T)$, and replacing the unique edge containing $v$ with an additional marking; here $n' \leq n+1$ and $d' \leq d$. Restriction of the map determines a surjective morphism
\[ \rho: \K^{\mathrm{iso}}(\T) \to \K^{\mathrm{iso}}(\T \smallsetminus v). \]
We claim that $\rho$ is smooth: indeed, if we let $n_v := |m^{-1}(v)|$ be the number of markings supported by $v$, and $d_v := \delta(v)$ be the map degree of the vertex $v$, we have a pullback diagram
\[\begin{tikzcd}
   &\K^{\mathrm{iso}}(\T) \arrow[r] \arrow[d, "\rho"] & \M_{0, n_v + 1}(\P^r, d_v) \arrow[d, "\mathrm{ev}"] \\\
   &\K^{\mathrm{iso}}(\T \smallsetminus v) \arrow[r, "\mathrm{ev}"] &\P^r
\end{tikzcd}.
\]
The bottom arrow and the rightmost arrow are evaluation maps at the two marked points which are glued to make $\T$. Since the rightmost evaluation at a single marked point is smooth, the map $\rho$ is smooth as well. By induction, the stratum $\K^{\mathrm{iso}}(\T\smallsetminus v)$ is irreducible. The fibers of $\rho$ are all isomorphic to $\M_{0, n_v}^*(\P^r,  d_v)$, which is irreducible. Since $\rho$ is a smooth morphism to an irreducible variety with irreducible fibers, we may conclude that $\K^{\mathrm{iso}}(\T)$ is irreducible.
\end{proof}

We will also need the irreducibility of the strata corresponding to $(1, n, d)$ graphs in $\Mbar_{1, n}(\P^r, d)$.
\begin{prop}\label{prop:conn_genus_one}
    Suppose $d > 0$. For any $(1, n, d)$-graph $\mathbf{G}$, the stratum $\K(\mathbf{G}) \subset \Mbar_{1, n}(\P^r, d)$ is irreducible.
\end{prop}
\begin{proof}
We will prove the connectedness of the moduli stack $\K^{\mathrm{iso}}(\mathbf{G})$ parameterizing maps in $\K(\mathbf{G})$ together with a chosen isomorphism of the dual graph with $\mathbf{G}$.

Using the same idea as in the proof of Proposition \ref{prop:conn_genus_zero}, we can reduce to the case where $\mathbf{G}$ is equal to its own core, so $\mathbf{G}$ is either a single vertex of genus one, or a cycle of genus zero vertices. When $\mathbf{G}$ is a single vertex, the result is the well-known fact $\M_{1, n}(\P^r, d)$ is irreducible.

Now suppose $\mathbf{G}$ is a cycle of genus zero vertices. Suppose we can find a vertex $v \in V(\mathbf{G})$ with map degree $d_v := \delta(v) \geq 2$. Setting $n_v := |m^{-1}(v)|$, we let $\mathbf{G}\smallsetminus v$ be the $(0, n - n_v +2, d - d_v)$-graph obtained by deleting $v$ from $\mathbf{G}$, and replacing the two incident edges by markings $*_1$ and $*_2$. Restriction defines a surjective morphism
\[\rho : \K^{\mathrm{iso}}(\mathbf{G}) \to \K^{\mathrm{iso}}(\mathbf{G} \smallsetminus v), \]
which fits into a fiber square
\begin{equation}\label{eqn:pullback_genus_one}
\begin{tikzcd}
&\K^{\mathrm{iso}}(\mathbf{G}) \arrow[d, "\rho"] \arrow[r] &\M_{0, n_v + 2}(\P^r, d_v) \arrow[d, "\mathrm{ev}_{*_1} \times \mathrm{ev}_{*_2}"] \\
&\K^{\mathrm{iso}}(\mathbf{G}\smallsetminus v) \arrow[r] &\P^r \times \P^r
\end{tikzcd}
\end{equation}
The morphism $\mathrm{ev_{*_1}} \times \mathrm{ev}_{*_2}$ is smooth. Therefore the restriction map $\rho$ is smooth. The fibers of $\rho$ are irreducible by Remark \ref{rem:two-pointed-maps}. Since $\K^{\mathrm{iso}}(\mathbf{G} \smallsetminus v)$ is irreducible, we can now conclude that $\K^{\mathrm{iso}}(\mathbf{G})$ is irreducible. 

The final case is when $\mathbf{G}$ does not contain any vertex $v$ with $d_v \geq 2$. Then we may suppose there is a vertex $v$ with $d_v = 1$ (if all vertices have degree $0$, then the claim reduces to the well-known irreducibility of dual graph strata in $\Mbar_{1, n}$). Now using this vertex $v$, we can replace $\P^r \times \P^r$ in (\ref{eqn:pullback_genus_one}) with the complement of the diagonal $(\P^r \times \P^r)\smallsetminus \Delta$, and we can replace $\K^{\mathrm{iso}}(\mathbf{G} \smallsetminus v)$ with the dense open where the two evaluations at the extreme markings do not coincide. Then the same argument goes through to prove that $\K^{\mathrm{iso}}(\mathbf{G})$ is indeed irreducible.
\end{proof}


\subsection{Strata as fiber products}
We now adapt our arguments to the setting of radially-aligned maps in genus one. The following result is the first part of Theorem \ref{thm:determination}.

\begin{thm}\label{conprop}
    The stratum $\widetilde{\mathcal{M}}(\mathbf{G},\rho)$ is irreducible. It is empty if and only if $d_{\min}(\mathbf{G}, \rho) = 1$.
\end{thm}
As outlined at the beginning of the section, the factorization property is automatically satisfied when the genus one core is not contracted. Hence $\widetilde{\M}(\mathbf{G},\rho)\to \K(\mathbf{G})$ is a torus fiber bundle. The stable map stratum $\K(\mathbf{G})$ is irreducible by Proposition \ref{prop:conn_genus_one}.

Now we assume that the genus one core is contracted. We recall that $\mathsf{r}$ denotes the contraction radius.
\begin{defn}For each $v\in \rho^{-1}(\mathsf{r})$, let $G_v$ be the tree that is the minimal subgraph consisting of all vertices $w$ such that $\rho(w)>\mathsf{r}$ and are connected to $v$ by a path, such that each vertex $u$ on the path has $\rho(u)>\mathsf{r}.$

We restrict the degree $\delta$ and marking function on $\mathbf{G}$ to $G_v$ and denote the marking function as $m_{G_v}: J\to V(G_v)$ for some $J\subset \{1,\dots, n\}$. Now we attach an additional leg $\star_v$ along $v$ by defining $m'_{G_v}: J\sqcup\{\star_v\}\to V(G_v)$ that extends $m_{G_v}$ with $\star_v\mapsto v$. The tuple $(G_v, \delta, m'_{G_v})$ defines a stable map dual graph that we denote $\mathbf{G}_v$. The additional marked point $\star_v$ is distinguished and may be indicated by writing $\mathbf{G}_v$ as a pair $(\mathbf{G}_v, \star_v)$. 


Let $\mathcal{K}(\mathbf{G}_v, \star_v)$ be the corresponding (genus-zero) stable map stratum. The marking $\star_v$ gives an evaluation map $\mathrm{ev}_{\star_v}: \mathcal{K}(\mathbf{G}_v, \star_v)\to \mathbb{P}^r$. Let $\mathcal{K}^*(\mathbf{G}_v, \star_v)$ be a fiber $\mathrm{ev}_{\star_v}^{-1}(p)$ for some fixed $p\in \mathbb{P}^r$.

\end{defn}

\begin{rem}Geometrically, consider the union of components $C_v$ corresponding to $v$ and those further from the core than $C_v$. They form a subcurve that is connected to the core by the node $\mu_{v}.$ The pair $(\mathbf{G}_v, \star_v)$ is the dual graph of the subcurve together with an extra marked point corresponding to $\mu_{v}.$ 

The following is an alternative way to describe $\mathbf{G}_v$. Consider $\mathbf{G}_{\geq \mathsf{r}}$ as the subgraph obtained from $\mathbf{G}$ by deleting all $\{v\in V(\mathbf{G}): \rho(v)<\mathsf{r}\}$ and their incident edges without modifying the degree and marked points on the remaining vertices. The decorated tree $\mathbf{G}_v$ is the connected component of $\mathbf{G}_{\geq \mathsf{r}}$ that contains $v$ together with an additional marked point $\star_v$.
\end{rem}

To return to the stratum $\widetilde{\M}(\mathbf{G},\rho),$ it is convenient to parametrize the linear dependency realized by the representative tangent vectors in the following way.

\begin{defn}
    Let
    \[\mathcal{V}_{\mathrm{tree}}^{*, \mathsf{F},\rho}(\mathbf{G}) \to \prod_{v\in \rho^{-1}(\mathsf{r})} \mathcal{K}^*(\mathbf{G}_v, \star_v)\] be the vector bundle whose fiber $V_{\boldsymbol{f}}$ over a tuple of maps $\boldsymbol{f} = (f_v)_{v \in \rho^{-1}(\mathsf{r})}$ is $\bigoplus_{v\in \rho^{-1}(\mathsf{r})}T_{\mu_v}C_v$.
    Let
\[\mathbb{P}^{\circ}\mathcal{V}_{\mathrm{tree}}^{*, \mathsf{F},\rho}(\mathbf{G}) \to \prod_{v\in \rho^{-1}(\mathsf{r})} \mathcal{K}^*(\mathbf{G}_v, \star_v)\] be the torus bundle whose fiber over $\boldsymbol{f}$ is the space of lines in $V_{\boldsymbol{f}}$ which are not contained in any coordinate subspace. Define $\widetilde{\mathcal{K}}^{*, \mathsf{F},\rho}_{\mathrm{tree}}(\mathbf{G})$ as the closed substack of $\mathbb{P}^{\circ}\mathcal{V}_{\mathrm{tree}}^{*, \mathsf{F},\rho}(\mathbf{G})$ consisting of pairs $(f,L)$ such that $\partial f(L) = 0.$
\end{defn}
\begin{lem}
    The space $\widetilde{\mathcal{K}}_{\mathrm{tree}}^{*, \mathsf{F}, \rho}(\mathbf{G})$ is irreducible and is empty if and only if $d_{\mathrm{min}}(\mathbf{G})=1.$
\end{lem}
\begin{proof}
    We again consider the moduli space $\widetilde{\mathcal{K}}_{\mathrm{tree}}^{*, \mathsf{F}, \rho}(\mathbf{G})_{\mathrm{iso}}$ parametrizing a point in $\widetilde{\mathcal{K}}^{*, \mathsf{F},\rho}_{\mathrm{tree}}(\mathbf{G})$ together with a chosen isomorphism between its dual graph and $\bigsqcup_{v\in \rho^{-1}(\mathsf{r})}(\mathbf{G}_v, \star_v)$. By the same argument as in the proof of Proposition \ref{prop:conn_genus_zero}, we may run induction on the number of vertices that are not in $\rho^{-1}(\mathsf{r})$ and reduce to checking irreducibility in the case where all vertices are on $\rho^{-1}(\mathsf{r}).$ Then, $\widetilde{\mathcal{K}}_{\mathrm{tree}}^{*, \mathsf{F}, \rho}(\mathbf{G})_{\mathrm{iso}}$ is isomorphic to $$\widetilde{\mathcal{K}}^{*, \mathsf{F},\rho}_{\mathrm{tree}}(\mathbf{G})_{\mathrm{iso}}\cong \left(\prod_{v\in \rho^{-1}(\mathsf{r})} \mathrm{Conf}^{\mathrm{val}(v)-1}(\mathbb{P}^1)\times \widetilde{\mathrm{Map}}^{*,\mathsf{F}, \rho}_{\mathbf{G}}(\mathbb{P}^1,\mathbb{P}^r)\right)/\prod_{v\in \rho^{-1}(\mathsf{r})} \mathrm{Aut}(C_v,*),$$
    where $\mathrm{Aut}(C_v,*)$ denotes the pointed automorphism group on each component ${C_v}$ for ${v\in \rho^{-1}(\mathsf{r})}$. Since we know that $\widetilde{\mathrm{Map}}^{*,\mathsf{F}, \rho}_{\mathbf{G}}$ is irreducible from Lemma \ref{mapcon}, and $\mathrm{Conf}^{\mathrm{val}(v)-1}(\mathbb{P}^1)$ are all irreducible, the quotient of their product is irreducible as well. The criterion for non-emptiness comes from that of $\widetilde{\mathrm{Map}}^{*,\mathsf{F}, \rho}_{\mathbf{G}}$. 
\end{proof}

We are in a position to complete the proof of Theorem \ref{conprop}. We first set up some notation.

\begin{defn}
    Let $\widetilde{\mathcal{M}}^{\mathrm{iso}}(\mathbf{G},\rho)$ be the moduli space parametrizing a point in $\widetilde{\mathcal{M}}(\mathbf{G},\rho)$ and an isomorphism between its dual graph and $(\mathbf{G},\rho).$ We have $\widetilde{\mathcal{M}}^{\mathrm{iso}}(\mathbf{G},\rho)/\mathrm{Aut}(\mathbf{G},\rho) = \widetilde{\mathcal{M}}(\mathbf{G},\rho).$
\end{defn}

\begin{defn}
    Let $\ell\in \{1,\dots, k\}$ be a level for the pair $(\mathbf{G},\rho).$ For $v\in \rho^{-1}(\ell),$ let $\star_v$ be the unique half-edge incident to $v$ that connects $v$ to $\rho^{-1}(\ell-1)$ and let $\mu_v\in C_v$ be the corresponding marking. Let $$\widetilde{\M}_{\rho^{-1}(\ell)}\to \prod_{v\in \rho^{-1}(\ell)}\mathcal{M}_{0, \mathrm{val}(v)}(\mathbb{P}^r,d_v)$$ be the total space of the $(\mathbb{C}^\star)^{|\rho^{-1}(v)|}/\mathbb{C}^\star$-bundle that parametrizes a tuple of genus zero maps together with a general line in the direct sum of the tangent spaces $\bigoplus_{v\in \rho^{-1}(\ell)}T_{\mu_v}C_{v}$ that is not contained in any coordinate hyperplane.

    The evaluation maps along $\mu_v$ define $\widetilde{\M}_{\rho^{-1}(\ell)}\to \prod_{v\in \rho^{-1}(\ell)}\mathbb{P}^r.$ This is a Zariski locally trivial fibration and we denote its fiber by $\widetilde{\M}^*_{\rho^{-1}(\ell)}$.
\end{defn}

\begin{proof}[Proof of Theorem \ref{conprop}]
    
    We prove the irreducibility of $\widetilde{\mathcal{M}}^{\mathrm{iso}}(\mathbf{G},\rho),$ which implies the irreducibility of $\widetilde{\M}(\mathbf{G},\rho).$

    When the core has positive degree, the factorization property is vacuous, and $\widetilde{\M}(\mathbf{G},\rho)$ is the total space of a torus bundle over $\mathcal{K}(\mathbf{G}).$ Its irreducibility follows from that of $\mathcal{K}(\mathbf{G}),$ which has been proven in Proposition \ref{prop:conn_genus_one}.

    Suppose the core has degree zero. We run induction on the number of levels outside of the contraction radius. When the whole curve is contracted (namely no contraction radius is present, and the morphism has degree zero), the stratum $\widetilde{\M}(\mathbf{G},\rho)$ is the product of $\mathbb{P}^r$ and some torus bundle over a dual graph stratum in $\Mbar_{1,n'},$ which is irreducible. Now suppose that the degree is positive, the only non-contracted components are on the contraction radius, and that there are no components outside the contraction radius. Let $(\mathbf{G}_c, \rho_c)$ be the radially-aligned graph obtained by deleting from $(\mathbf{G}, \rho)$ the vertices at the contraction radius, and replacing the incident edges with markings. The forgetful map $\widetilde{\M}^{\mathrm{iso}}(\mathbf{G},\rho)\to \widetilde{\M}^{\mathrm{iso}}(\mathbf{G}_c, \rho_c)$ is a morphism between smooth varieties, with fibers $\mathcal{K}_{\mathrm{tree}}^{*, \mathsf{F}, \rho}(\mathbf{G}).$ Since these fibers have constant dimension, we can see that the morphism is flat. Since the base and fiber are both irreducible, $\widetilde{\M}^{\mathrm{iso}}(\mathbf{G},\rho)$ is irreducible as well.

    For the inductive step, one applies a similar argument to the map that forgets the outermost level $\ell$ of $(\mathbf{G}, \rho)$, replacing the fiber $\mathcal{K}_{\mathrm{tree}}^{*, \mathsf{F}, \rho}(\mathbf{G})$ with $\widetilde{\M}^*_{\rho^{-1}(\ell)}$.
\end{proof}
\section{Topology of dual complexes}\label{sec:dualcomplex}

In this section, we will assemble the preceding results to describe the dual complex $\Delta_{0, n}(d)$ of the boundary divisor in $\Mbar_{0, n}(\P^r, d)$ and the dual complex $\Delta_{1, n}(d)$ of the boundary divisor in $\widetilde{\M}_{1, n}(\mathbb{P}^r, d)$. Then we will prove the second main theorem of this paper, recalled below.

\begin{customthm}{B}
    Fix $r\geq 1$. The dual complex $\Delta_{0, n}(d)$ of \[\Mbar_{0, n}(\mathbb{P}^r, d) \smallsetminus \M_{0, n}(\mathbb{P}^r, d) \] for $d>0$ as well as the dual complex $\Delta_{1, n}(d)$ of 
    \[\widetilde{\M}_{1, n}(\mathbb{P}^r, d) \smallsetminus \M_{1, n}(\mathbb{P}^r, d)   \] for $d>1$ are contractible. In particular, the reduced homology groups of the dual complexes vanish.
\end{customthm}

Our proof of this theorem proceeds as follows: for each $g \geq 0$, we define a symmetric $\Delta$-complex $\Delta_{g, n}^{\mathrm{vir}}(d)$ which tracks the poset structure of the locus of maps from singular curves in the Kontsevich moduli space $\Mbar_{g, n}(\mathbb{P}^r, d)$. We then describe a deformation retract of $\Delta_{g, n}^{\mathrm{vir}}(d)$ to a point for all $d > 0$. When $g = 0$, we have that $\Delta_{0, n}(d) = \Delta_{0, n}^{\mathrm{vir}}(d)$, so the deformation retraction we construct proves the first half of Theorem \ref{thm:mainthm}. When $g = 1$, we show that the dual complex $\Delta_{1, n}(d)$ is homeomorphic to a subspace of $\Delta_{1, n}^{\mathrm{vir}}(d)$, and that this subspace is preserved by our deformation retract.

\subsection{The virtual dual complex $\Delta_{g, n}^{\mathrm{vir}}(d)$}\label{sec:deltavir}

Associated to a stable $(g, n, d)$-graph $\mathbf{G}$ is a cell
\[ \kappa_{\mathbf{G}} := \left\{\ell : E(\mathbf{G}) \to \mathbb{R}_{\geq 0} \mid \sum_{e \in E(\mathbf{G})} \ell(e) = 1\right\}.  \]
A morphism $\mathbf{G} \to \mathbf{G}'$ in the category $\Gamma_{g, n}(d)$ induces an inclusion of cells
\[\kappa_{\mathbf{G}'} \to \kappa_{\mathbf{G}}, \]
so that we get a functor
\[\Gamma_{g, n}(d)^{\mathrm{op}} \to \mathsf{Spaces} \]
to the category $\mathsf{Spaces}$ of topological spaces.
\begin{defn}\label{def:vir_dual_complex}
    The \textit{virtual dual complex} of the stable maps compactification $\mathcal{M}_{g,n}(\mathbb{P}^r,d)\subset \Mbar_{g,n}(\mathbb{P}^r,d)$ is the colimit of the above functor.
\end{defn}
Definition \ref{def:vir_dual_complex} is well-suited for studying the topology of $\Delta_{g, n}^{\mathrm{vir}}(d)$. To compare it with actual dual complexes of normal crossings compactifications, it is worthwhile to describe $\Delta_{g, n}^{\mathrm{vir}}(d)$ as a \textit{symmetric $\Delta$-complex} following \cite[\S 3.5]{cgp}. Recall that a symmetric $\Delta$-complex $X$ is a functor
\[X: I^{\mathrm{op}} \to \mathsf{Set} \]
where 
\begin{itemize}
\item $I$ is the category consisting of sets of the form $[p] := \{0, \ldots, p\}$ for $p \geq 0$ and the set $[-1]:= \varnothing$, and whose morphisms are injections;
\item $\mathsf{Set}$ is the category of sets.
\end{itemize}
There is a \textit{geometric realization functor} from symmetric $\Delta$-complexes to topological spaces; c.f. \cite[(3.1.1)]{cgp}). Strictly speaking, Definition \ref{def:vir_dual_complex} defines the geometric realization of $\Delta_{g, n}^{\mathrm{vir}}(d)$; we will use the same notation for the functor and the geometric realization, since ultimately we are interested in the topology of the geometric realization. For the functorial description, first set, for $p \geq -1$,
\[ \Gamma_{g, n}^p(d) \subset \Gamma_{g, n}(d) \]
for the subgroupoid consisting of graphs with $p+1$ edges. Write $\mathrm{Iso}(\Gamma_{g, n}^p(d))$ for the finite set of isomorphism classes in this groupoid. Now set
\begin{equation}\label{eqn:sym_delta_vir}
\Delta_{g, n}^{\mathrm{vir}}(d)[p] := \{ (\mathbf{G}, \tau) \mid \mathbf{G} \in \mathrm{Iso}(\Gamma_{0, n}^p(d)) \text{ and }\tau:E(\mathbf{G}) \to [p] \text{ is a bijection} \}/\sim,
\end{equation}
where $(\mathbf{G}, \tau)\sim (\mathbf{G}', \tau')$ if and only if there is an isomorphism $\mathbf{G} \to \mathbf{G}'$ in $\Gamma_{g, n}(d)$ which respects the edge-labellings $\tau$ and $\tau'$. We write $[\mathbf{G}, \tau]$ to denote the equivalence class of an edge-labelled pair. Given an injection $\iota: [q] \to [p]$ in $I$, the corresponding map
\[ \Delta_{g, n}^{\mathrm{vir}}(d)[p] \to \Delta^{\mathrm{vir}}_{g, n}(d)[q] \]
is given on an edge-labelled graph $[\mathbf{G}, \tau]$ by contracting the edges $\tau^{-1}(j)$ for $j \in [p] \smallsetminus \iota([q])$, and then taking the unique ordering of the remaining edges which is compatible with $\tau$.

\begin{rem}
    Let $\Mfr_{g,n}$ be the stack of $n$-pointed prestable genus $g$ curves. Let $\Mfr_{g,n}(d)\to \Mfr_{g,n}$ be the finite map that parametrizes in addition degree labelings of irreducible components which are non-negative and sum to $d$. Finally, let $\Mfr_{g,n}^{\mathrm{st}}(d)\subset \Mfr_{g,n}(d)$ be the open substack consisting of degree-labeled prestable curves which satisfy the same combinatorial stability condition as those for stable maps. The stack $\Mfr_{g,n}^{\mathrm{st}}(d)$ inherits a logarithmic structure by pulling back the one on $\Mfr_{g,n}$. The virtual dual complex $\Delta^{\mathrm{vir}}_{g,n}(d)$ agrees with the dual complex of the divisor of singular curves in the logarithmic stack $\Mfr_{g,n}^{\mathrm{st}}(d)$. The choice of the term `virtual' is justified by the fact that $\Mbar_{g,n}(\mathbb{P}^r,d)\to \Mfr_{g,n}^{\mathrm{st}}(d)$ is virtually smooth.
\end{rem}

\begin{rem}
Two clarifying remarks on $\Delta_{g, n}^{\mathrm{vir}}(d)$ are in order:
    \begin{itemize}
        \item Recall that the compactification $\mathcal{M}_{g,n}(\mathbb{P}^r,d)\subset \Mbar_{g,n}(\mathbb{P}^r,d)$ fails to be of normal crossings--- or indeed even irreducible---for $g\geq 1$. Therefore, it does not make sense to talk of the `dual complex' of the stable maps compactification when $g\geq 1$, hence the term `virtual' in our definition.
        \item For all $r$, curve classes in $\mathbb{P}^r$ are classified by the degree in $H_2(\mathbb{P}^r)\cong \mathbb{Z}$, so that the $(g,n,d)$-graphs do not depend on the target dimension $r$. Therefore, we suppress the $r$ in our notation $\Delta_{g, n}^{\mathrm{vir}}(d)$.
    \end{itemize}
\end{rem}

 Now we will explain why the dual complex $\Delta_{0, n}(d)$ of the normal crossings compactification
\[\M_{0, n}(\mathbb{P}^r, d) \hookrightarrow \Mbar_{0, n}(\mathbb{P}^r, d)\]
coincides with $\Delta_{0, n}^{\mathrm{vir}}(d)$.

\begin{prop}\label{prop:genuszero_dualcomplex}
The dual complex of the divisor of singular curves in the Kontsevich space of stable maps $\Mbar_{0, n}(\mathbb{P}^r, d)$ coincides with $\Delta_{0, n}^{\mathrm{vir}}(d)$ for $r \geq 2$ and $d > 0$.
\end{prop}
\begin{proof}
Let $\Delta_{0, n}(d)$ denote the dual complex of the divisor $D$ in question, and set $X = \Mbar_{0, n}(\mathbb{P}^r, d)$. Intersections of boundary divisors on $X$ have long been well-understood: see e.g. \cite[\S 1]{opretorus}. The codimension $p$ strata of the boundary are in bijection with the elements of $\mathrm{Iso}(\Gamma_{0, n}^{p-1}(d))$. Using the description of dual complexes in \cite[Definition 5.2]{cgp}, we can understand the $p$-cells of $\Delta_{0, n}(d)$ by
\[\Delta_{0, n}(d)[p] = \pi_0\left(\underbrace{\tilde{D} \times_X \cdots \times_X \tilde{D}}_{p + 1} \smallsetminus \{ (z_0, \ldots, z_{p} ) \mid z_i = z_j \text{ for some }i \neq j \}\right), \]
where $\tilde{D} \to D$ is the normalization of the boundary divisor. A point of the fiber product can be thought of as a point $x$ on a codimension $p +  1$ stratum, together with an ordering $\sigma$ of the branches of $D$ at $x$. Stratifying by the isomorphism types of the $(0,n,d)$-graphs, we can thus write the above fiber product as
\[\bigcup_{\mathbf{G} \in \mathrm{Iso}(\Gamma_{0, n}^{p}(d))} \K^{\mathrm{iso}}(\mathbf{G})/A_{\mathbf{G}}\]
where $A_{\mathbf{G}}$ is the finite subgroup of $\mathrm{Aut}(\mathbf{G})$ which preserves the edges of $\mathbf{G}$. Since each $\K^{\mathrm{iso}}(\mathbf{G})/A_{\mathbf{G}}$ is irreducible (by the proof of Proposition \ref{prop:conn_genus_one}) and the intersection of two strata corresponding to distinct isomorphism classes of $(0,n,d)$-graphs is clearly empty, this gives a natural bijection between $\Delta_{0, n}(d)[p]$ and $\Delta_{0, n}^{\mathrm{vir}}(d)[p]$.

The face maps in $\Delta_{0, n}(d)$ are opposite to inclusions of closures of boundary strata; therefore, they correspond to smoothing nodes of the source curves of stable maps. On the level of dual graphs, smoothing nodes corresponds to contracting edges. As such, there is a natural isomorphism of functors
\[ \Delta_{0, n}(d) \cong \Delta_{0, n}^{\mathrm{vir}}(d), \]
as desired.
\end{proof}

\subsection{Contractibility of the virtual dual complex}

We need some combinatorial definitions before the proof.

\begin{defn}
Given a stable $(g, n, d)$-graph $\mathbf{G}$, we call an edge $e \in E(\mathbf{G})$ a $1$-\textit{end} if it separates the graph into two connected components, one of which consists of a single vertex $v$ with $w(v) = 0$, $|m^{-1}(v)| = 0$, and $\delta(v) = 1$. The vertex $v$ is called a $1$-\textit{end vertex}. Among the graphs $\tilde{\mathbf{G}}$ such that $\mathbf{G}$ is obtained from $\tilde{\mathbf{G}}$ by contracting a sequence of $1$-ends, there is up to isomorphism a unique graph $\mathbf{G}^{\mathrm{sp}}$ which has the maximal number of $1$-ends. We call $\mathbf{G}^{\mathrm{sp}}$ the \textit{sprouting} of $\mathbf{G}$.
\end{defn}

\begin{figure}
    \centering
    \includegraphics[scale=1.3]{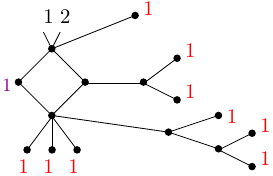}
    \caption{The sprouting $\mathbf{G}^{\mathrm{sp}}$ of the graph $\mathbf{G}$ from Figure \ref{fig:vir-comb-type-exmp}.}
    \label{fig:sprout-exmp}
\end{figure}

\begin{rem}
    More explicitly, the graph $\mathbf{G}^{\mathrm{sp}}$ is obtained from $\mathbf{G}$ via the following recipe: for each $v\in V(\mathbf{G})$ which is not a $1$-end vertex, add $\delta(v)$ copies of vertices and connect each of them to $v$ via a single edge; on the new graph, replace $\delta(v)$ by zero and assign degree one to each of the $\delta(v)$ added vertices.

    In particular, we may perform the operation on $\mathbf{G}_{\varnothing}$, the unique $(g,n,d)$-graph that only has a single vertex. Observe that for any $(g,n,d)$-graph $\mathbf{G}$, both cells $\kappa_{\mathbf{G}}$ and $\kappa_{\mathbf{G}_{\varnothing}^{\mathrm{sp}}}$ are faces of the cell $\kappa_{\mathbf{G}^{\mathrm{sp}}}$. The deformation retract of interest consists of certain linear homotopies in the cells $\kappa_{\mathbf{G}^{\mathrm{sp}}}$, as we now describe.
\end{rem}

\begin{thm}\label{thm:virtual-contraction}
For all $g, n, d$ with $d > 0$, the complex $\Delta_{g, n}^{\mathrm{vir}}(d)$ is contractible.
\end{thm}
\begin{proof}
We will define a deformation retract \[\Delta_{g, n}^{\mathrm{vir}}(d) \times [0, 1] \to \Delta_{g, n}^{\mathrm{vir}}(d),\]
working cell-by-cell. We define
\[ f_{\mathbf{G}}: \kappa_{\mathbf{G}} \times [0, 1] \to \Delta_{g, n}^{\mathrm{vir}}(d) \]
by
\[f_{\mathbf{G}}(\ell) = (\mathbf{G}^{\mathrm{sp}}, \ell_t), \]
where for $e \in E(\mathbf{G}^{\mathrm{sp}})$ we define
\[ \ell_t(e) = (1-t)\ell(e) + t/d \]
if $e$ is a $1$-end, and
\[\ell_t(e) = (1 - t) \ell(e) \]
otherwise. Then at $t = 1$, the image of $f_{\mathbf{G}}$ recovers $\mathbf{G}_{\varnothing}^{\mathrm{sp}}$with length $1/d$ on each of the $d$ edges. 

We claim that the maps $f_{\mathbf{G}}$ glue together to form a deformation retract of $\Delta_{1, n}^{\mathrm{vir}}(d)$. To do this, we have to prove that whenever $\psi: \mathbf{G} \to \mathbf{G}'$ is a morphism in $\Gamma_{g, n}(d)$, the diagram
\[
\begin{tikzcd}
\kappa_{\mathbf{G}} \times [0, 1] \arrow[r, "f_{\mathbf{G}}"]            & \Delta_{g, n}^{\mathrm{vir}}(d) \\
\kappa_{\mathbf{G}'} \times [0, 1] \arrow[ru, "f_{\mathbf{G}'}"'] \arrow[u, "\psi^*"] &  
\end{tikzcd}
\]
commutes. As morphisms in $\Gamma_{g,n}(d)$ are compositions of edge contractions and isomorphisms, it suffices to assume that $\psi$ is one of the two types. 

First consider the case where $\psi$ is a contraction of an edge $e$. The map $\psi$ induces a bijection $E(\mathbf{G})\smallsetminus \{e\}\xrightarrow{\cong} E(G')$, and one can check from the construction that for all $\tilde{e}\in E(\mathbf{G})\smallsetminus \{e\}$, $\ell_{t}(\tilde{e}) = \ell_t(\psi(\tilde{e}))$. Therefore, it suffices to check that the two maps agree for the coordinate $\ell_t(e)$ on $\kappa_{\mathbf{G}}\times [0,1]$. Then $\psi^*$ simply extends the metric on $\mathbf{G}'$ to one on $\mathbf{G}$ by setting the length of $e$ to $0$. If $e$ is not a $1$-end, then its length remains $0$ when applying $f_{\mathbf{G}}$ for all $t > 0$, so the diagram commutes in this case. When $e$ is a $1$-end, then $\mathbf{G}^{\mathrm{sp}} = (\mathbf{G}')^{\mathrm{sp}}$, and $e$ has length $t/d$ when applying both $f_{\mathbf{G}} \circ \psi^*$ and $f_{\mathbf{G}'}$, so the diagram also commutes in this case. 

When $\psi$ is an isomorphism of graphs, the substance of the claim is that the metric graphs $(\mathbf{G}'^{\mathrm{sp}}, \ell'_t)$ and $(\mathbf{G}^{\mathrm{sp}}, (\psi^*\ell')_t)$ will also be isomorphic: this is true because there exists an isomorphism $\mathbf{G}^\mathrm{sp} \to (\mathbf{G}')^{\mathrm{sp}}$ which extends $\psi$. Hence the maps $f_{\mathbf{G}}$ glue to give a deformation retract of $\Delta_{g, n}^{\mathrm{vir}}(d)$. 
\end{proof}

\subsection{Description of $\Delta_{1, n}(d)$ as a symmetric $\Delta$-complex}

We are now ready to describe the dual complex $\Delta_{1, n}(d)$ of $\M_{1, n}(r, d)$, as a symmetric $\Delta$-complex. Recall that formally, a symmetric $\Delta$-complex $X$ is a functor $I^\mathrm{op} \to \mathsf{Set}$, where $I$ is the category whose morphisms are injections and whose objects are the finite sets
\[[p] := \{0, \ldots, p\}, \]
for $p \geq -1$; we formally set $[-1] = \varnothing$. We define
$\Delta_{1, n}(d)[p]$ to be the set of isomorphism classes of pairs $(\mathbf{G}, \omega)$ where
\begin{itemize}
    \item $\mathbf{G}$ is an object in the groupoid $\tilde{\Gamma}_{1, n}^\rad(d)$ (strictly speaking, the graph $\mathbf{G}$ comes with a radial alignment $\rho$ such that $d_{\min}(\mathbf{G}, \rho) > 1$ as in Definition \ref{def:minimal-radius}; we will suppress $\rho$ in the notation), and
    \item $\omega: C(\mathbf{G}) \sqcup \{1, \ldots, k \} \to [p]$ is a bijection, where $k$ is equal to the length of $\mathbf{G}$.
\end{itemize} 
Two such pairs $(\mathbf{G}, \omega)$ and $(\mathbf{G}', \omega')$ are isomorphic if there exists an isomorphism $\psi: \mathbf{G} \to \mathbf{G}'$ in $\tilde{\Gamma}_{1, n}^\rad(d)$  such that for every edge $e$ in the core of $\mathbf{G}$, we have $\omega(e) = \omega'(\psi(e))$, and such that for all $i \in \{1, \ldots, k\}$, we have $\omega(i) = \omega'(i)$. Given an injection $\iota: [q] \to [p]$, we define
\[\Delta_{1, n}(d)(\iota) : \Delta_{1, n}(d)[p] \to \Delta_{1, n}(d)[q]\]
as follows: given $(\mathbf{G}, \omega)$ in $\Delta_{1, n}(d)[p]$, perform radial merges and core edge contractions along those $\omega^{-1}(j)$ for $j \in [p] \smallsetminus \iota([q])$ and then relabel according to the ordering provided by $\omega$. As stated in Remark \ref{rem:face_closure}, the quantity $d_{\mathrm{min}}(\mathbf{G}, \rho)$ can only increase upon taking contractions or radial merges, so $\Delta_{1, n}(d)$ is closed under taking faces. 

The following proposition completes the proof of Theorem \ref{thm:determination}.

\begin{prop}
The symmetric $\Delta$-complex $\Delta_{1, n}(d)$ described above coincides with the dual complex of the compactification \[\M_{1, n}(\mathbb{P}^r, d) \hookrightarrow \widetilde{\M}_{1, n}(\mathbb{P}^r, d). \]
\end{prop}
\begin{proof}
Let $X = \widetilde{\M}_{1, n}(\mathbb{P}^r, d)$, set \[D = \widetilde{\M}_{1, n}(\mathbb{P}^r, d) \smallsetminus \M_{1, n}(\mathbb{P}^r, d), \]
and write $\Delta(D)$ for the dual complex of $D$; we want to show that $\Delta(D) = \Delta_{1, n}(d)$. The proof that
\[ \Delta(D)[p] = \Delta_{1, n}(d)[p] \]
for each $p \geq 0$ is almost exactly the same as that of Proposition \ref{prop:genuszero_dualcomplex}, since we have proved the irreducibility of the strata $\widetilde{\M}(\mathbf{G}, \rho)$. In this case, the analytic branches at a point in $\widetilde{\M}(\mathbf{G}, \rho)$ correspond to the core edges of $\mathbf{G}$ together with the levels of $\rho$. Because Theorem \ref{thm: PosetStructure} implies that the face maps in $\Delta(D)$ reflect morphisms in the category $\tgamma^{\rad}_{1, n}(d),$ we  upgrade the bijection $\Delta(D)[p] = \Delta_{1,n}(d)[p]$ to an isomorphism of functors $\Delta(D) \cong \Delta_{1, n}(d)$.
 \end{proof}

\subsection{The dual complex as a topological space} The geometric realization of $\Delta_{1, n}(d)$ can be described as follows: given $\mathbf{G} \in \mathrm{Ob}(\tgamma_{1, n}^\rad(d))$ with $\ell(\mathbf{G}) = k$ and such that $|C(\mathbf{G})| + k = p + 1$, define a cell
\[\sigma_{\mathbf{G}} = \left\{(\mathbf{G}, \ell) \mid \ell: C(\mathbf{G}) \sqcup \{1, \ldots, k\} \to \mathbb{R}_{\geq 0}, \, \sum_{e \in C(\mathbf{G})} \ell(e) + \sum_{i = 1}^{k} \ell(i) = 1\right\}.  \]
A morphism $\mathbf{G} \to \mathbf{G}'$ in the category $\tgamma_{1, n}^\rad(d)$ induces in a natural way an inclusion of cells
\[ \sigma_{\mathbf{G}'} \to \sigma_{\mathbf{G}}.  \]
In this way, we can view the association of a cell to a graph as a functor  \[\sigma: \tilde{\Gamma}_{1, n}(d)^{\mathrm{op}} \to \mathsf{Spaces}. \]
Then the geometric realization of $\Delta_{1, n}(d)$ is identified with the colimit of this functor. Concretely, a point on $\Delta_{1, n}(d)$ can be thought as a pair $(\hat{\mathbf{G}}, \ell)$ where $\hat{\mathbf{G}}$ is the canonical subdivision of $\mathbf{G} \in \Ob(\tgamma_{1, n}^\rad(d))$, and $\ell: E(\hat{\mathbf{G}}) \to \mathbb{R}_{\geq 0}$ is a metric of total length $1$, such that $\ell$ takes the same value on any two edges of $\hat{\mathbf{G}}$ which lie over the same edge in $P_k$ under the canonical map $\hat{\mathbf{G}} \to P_k$. Equivalently, we can think of the metric $\ell$ as a function from $C(\mathbf{G}) \cup E(P_k) \to \mathbb{R}_{\geq 0}$ of total length $1$. The metric on the edges of $P_k$ induces the metric on the non-core edges of $\hat{\mathbf{G}}$ by setting the length of an edge $e$ to be 
\[ \frac{\ell(\hat{\rho}(e))}{|\hat{\rho}^{-1}(e) |}; \]
see Figure \ref{fig:point-example}.

\begin{figure}
    \centering
    \includegraphics[scale=1.3]{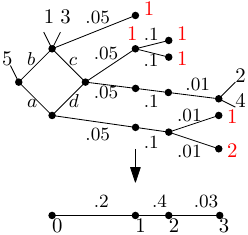}
    \caption{An example of a point in the cell $\sigma_{\mathbf{G}}$, where $\mathbf{G}$ is the graph from Figure \ref{fig:comb-type-exmp}. The metric on the edges of $\hat{\mathbf{G}}$ is induced by a metric on the edges of $P_3$, together with a metric on the core edges of $\mathbf{G}$; in this example, $a, b, c, d \in \mathbb{R}_{\geq 0}$ can be chosen to be any numbers satisfying $a + b + c + d = .37$. }
    \label{fig:point-example}
\end{figure}

Performing a radial merge or a core edge contraction corresponds to taking a face of a cell, by setting some edge lengths to $0$.
\subsection{The dual complex as a subspace of $\Delta_{1, n}^{\mathrm{vir}}(d)$} We are now ready to prove that the dual complex $\Delta_{1, n}(d)$ is contractible. In this section, we will work with the geometric realizations of $\Delta_{1, n}(d)$ and $\Delta_{1, n}^{\mathrm{vir}}(d)$. We begin by extending Definition \ref{def:minimal-radius} to points of $\Delta_{1, n}^{\mathrm{vir}}(d)$.
\begin{defn}
Suppose $(\mathbf{G}, \ell) \in \Delta_{1, n}^{\mathrm{vir}}(d)$. For each vertex $v \in V(\mathbf{G})$ outside the core, the distance $\mathrm{dist}(v)$ of $v$ from the core is well-defined. Define a radial alignment $\rho_{\ell}$ on $\mathbf{G}$ by ordering the vertices by $\mathrm{dist}(v)$. We call $\rho_{\ell}$ the \textit{canonical radial alignment} of the metric graph $(\mathbf{G}, \ell)$.
\end{defn}

Define a subspace
\[Z \hookrightarrow \Delta_{1, n}^{\mathrm{vir}}(d) \]
by
\[Z = \{(\mathbf{G}, \ell) \mid d_{\mathrm{min}}(\mathbf{G}, \rho_{\ell}) > 1\}. \] Then we have a homeomorphism $\Delta_{1, n}(d) \to Z$ which takes $(\hat{\mathbf{G}}, \ell)$ to $(\mathbf{G}, \hat{\ell})$, where $\hat{\ell}$ is obtained by adding the edge lengths across bivalent vertices which are smoothed; see Figure \ref{fig:embedding-exmp}.

\begin{figure}
    \centering
    \includegraphics[scale=1.3]{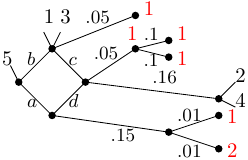}
    \caption{The image of the point of $\Delta_{1, 5}(7)$ from Figure \ref{fig:point-example} under the embedding $\Delta_{1, 5}(7) \hookrightarrow \Delta_{1, 5}^{\mathrm{vir}}(7)$.}
    \label{fig:embedding-exmp}
\end{figure}

\begin{prop}\label{prop:subspace_preserved}
The deformation retract of $\Delta_{1, n}^{\mathrm{vir}}(d)$ described in the proof of Theorem \ref{thm:virtual-contraction} preserves the subspace $Z$ defined above. In particular, $\Delta_{1, n}(d)$ is contractible.
\end{prop}
\begin{proof}
We will show that for all $t \in [0, 1]$ and points $(\mathbf{G}, \ell)$ in $\Delta_{1, n}^{\mathrm{vir}}(d)$, we have
\[ d_{\mathrm{min}}(\mathbf{G}^{\mathrm{sp}}, \ell_t) \geq d_{\mathrm{min}}(\mathbf{G}, \ell) = d_{\mathrm{min}}(\mathbf{G}^{\mathrm{sp}}, \ell_0).\]
For this it suffices to show that for $v, w \in V(\mathbf{G}^{\mathrm{sp}})$ with $\delta(v) = \delta(w) = 1$, then if $\rho_{\ell_0}(v) \leq \rho_{\ell_0}(w)$ we also have $\rho_{\ell_t}(v) \leq \rho_{\ell_t}(w)$ for all $t > 0$. In this situation, $v$ and $w$ are necessarily $\delta$-degree one vertices with no markings, valence $1$, and genus zero. Therefore the minimal path from each vertex to the core of $\mathbf{G}^{\mathrm{sp}}$ consists of a sequence of edges which are not $1$-ends, followed by a single $1$-end. Let $e_1$ be the $1$-end supporting $v$ and $e_2$ be the $1$-end supporting $w$, and let $p_v(t)$ be the length of the path from the core to $v$ at time $t$; define $p_w(t)$ similarly. Then $p_v(t) \leq p_w(t)$ if and only if $\rho_{\ell_t}(v) \leq \rho_{\ell_t}(w)$. Notice that
\[ p_v(t) = (1 - t) \ell(e_1) + t/d + (1 - t)(p_v(0) - \ell(e_1)) = t/d + (1 - t)p_v(0) \]
by the definition of the homotopy. Similarly,
\[ p_w(t) = t/d + (1 - t)p_w(0). \]
Since $p_v(0) \leq p_w(0)$, we can conclude that $p_v(t) \leq p_w(t)$ for all $t \geq 0$, as we wanted to show. 
\end{proof}

\bibliographystyle{alpha}
\bibliography{vz.bib}

\end{document}